\documentclass[3p,times]{elsarticle}

\usepackage{mathpazo}
\usepackage{amssymb, amsmath,amsthm, mathrsfs}
\usepackage{color}
\usepackage[colorlinks=true, allcolors=blue]{hyperref}
\usepackage{mathtools}
\usepackage{cancel}
\newtheorem{theorem}{Theorem}
\newtheorem{proposition}[theorem]{Proposition}
\newtheorem{corollary}[theorem]{Corollary}
\newtheorem{lemma}[theorem]{Lemma}
\theoremstyle{definition}
\newtheorem{remark}{Remark}

\newcommand{\C}{\mathbb{C}}
\newcommand{\R}{\mathbb{R}}
\newcommand{\Z}{\mathbb{Z}}

\newcommand{\Hb}{\mathbf{H}}
\newcommand{\Lb}{\mathbf{L}}

\newcommand{\Cc}{\mathcal{C}}

\newcommand{\Wc}{\mathcal{W}}

\newcommand{\gk}{\mathfrak{g}}

\newcommand{\hf}{\widehat{f}}
\newcommand{\hpsi}{\widehat{\psi}}

\newcommand{\ellb}{\mathbf{\ell}}

\DeclarePairedDelimiterX{\norm}[1]{\lVert}{\rVert}{#1}

\begin{document}

\begin{frontmatter}



%% e.g. \dochead{17th International Conference on Dynamical Processes in Excited States of Solids}

\title{On Generalizations of the Nonwindowed Scattering Transform}
\tnotetext[tn]{This work was partially supported by the National Science Foundation [A.C. and M.H by grant DMS-1845856; A.L. and M.H. by grant DMS-1912906], the National Institutes of Health [M.H. by grant NIGMS-R01GM135929], and the Department of Energy [M.H. by grant DE-SC0021152].}


\author[msu]{Albert Chua\corref{coralbert}}
\ead{chuaalbe@msu.edu}
\author[msu,cmse,q]{Matthew Hirn}
\ead{mhirn@msu.edu}
\author[uu]{Anna Little}
\ead{little@math.utah.edu}
\cortext[coralbert]{Corresponding author}
\address[msu]{Department of Mathematics, Michigan State University, East Lansing, MI, 48824 USA}
\address[cmse]{Department of Computational Mathematics, Science \&
Engineering, Michigan State University, East Lansing, MI, 48824 USA}
\address[q]{Center for Quantum Computing, Science \& Engineering
Michigan State University, East Lansing, MI, 48824 USA}
\address[uu]{Department of Mathematics and the Utah Center For
Data Science, University of Utah, Salt Lake City, UT, 84112 USA}

\begin{abstract}
In this paper, we generalize finite depth wavelet scattering transforms, which we formulate as $\Lb^q(\mathbb{R}^n)$ norms of a cascade of continuous wavelet transforms (or dyadic wavelet transforms) and contractive nonlinearities. We then provide norms for these operators, prove that these operators are well-defined, and are Lipschitz continuous to the action of $C^2$ diffeomorphisms in specific cases. Lastly, we extend our results to formulate an operator invariant to the action of rotations $R \in \text{SO}(n)$ and an operator that is equivariant to the action of rotations of $R \in \text{SO}(n)$. 
\end{abstract}


\end{frontmatter}

%% write the stupid intro about quantum chemistry and add a bit about scattering moments
%% fix up q = 1 proof and check over again
%% check the condition for rotation invariance one more time.

\section{Introduction}
In recent years, convolutional neural networks have shown strong performance on various vision tasks like image classification \cite{Alexnet, VGG, Googlenet, resnet}. The main reason for this is that they are able to capture information at multiple scales through the use of convolutions and pooling. However, the exact method in which these networks use this information is not understood very well. 

In \cite{mallat2012group}, the author proposed a formulation for a simpler model for a convolutional neural network through the use of handcrafted filters, wavelets, and a series of cascading wavelet transforms. This model, called the scattering transform, and its extensions have shown success in vision tasks, quantum chemistry, manifold learning, and graph-related tasks \cite{graph, geometric, QMWavelet1, QMWavelet2, QMWavelet3}.

We first provide a review of scattering transforms to motivate this paper. Let  $\phi:\mathbb{R}^n \to \mathbb{R}$ be a low pass filter ($\hat{\phi}(0) \neq 0$) , $\psi:\mathbb{R}^n \to \mathbb{C}$ a suitable mother wavelet ($\hat{\psi}(0) = 0$), and $G^{+}$ be a set of “positive” rotations with determinant $1$. Define a set of rotations and dilations by
\begin{equation}\Lambda_J := \{\lambda = 2^{j}r : r \in G^{+}, j > -J \} \text{ if } J \neq \infty
\end{equation}
and 
\begin{equation}\Lambda_\infty := \{2^{j}r: r \in G^{+}, \,j \in \mathbb{Z}\}.
\end{equation}
Let $\lambda = 2^j r \in \Lambda_J$. Consider the operator
\begin{equation}
U[\lambda] = \left|\int_{\mathbb{R}^n} f(u) 2^{nj}\psi(2^jr^{-1}(x-u)) \, du \right|
\end{equation}
For a tuple of rotations and dilations in $\Lambda_J$, define a path of length $m$ as the tuple $p := (\lambda_1, \ldots, \lambda_m)$ and let $\mathcal{P}_J$ be the set of all finite paths. The scattering propagator for $f \in \Lb^2(\mathbb{R}^n)$ and $p \in \mathcal{P}_J$ is  
\begin{equation}
U[p]f := U[\lambda_m] \cdots U[\lambda_1] f,
\end{equation}
which gathers high frequency information via a cascade of wavelet transforms and nonlinearities. The scattering operator is 
\begin{equation}
\overline{S}{f}(p) = \frac{1}{\mu_p} \int_{\mathbb{R}^n} U[p]f(x) \, dx
\end{equation}
with $\mu_p := \int_{\mathbb{R}^n} U[p]\delta(x) \,dx.$ Additionally, to aggregate features similar to pooling, the author of \cite{mallat2012group} define the scattering operator for $f \in \Lb^2(\mathbb{R}^n)$ and $p \in \mathcal{P}_J$ as
\begin{equation}
S_J[p]f(x) = \int_{{\R}^n} U[p]f(u) 2^{-nJ}\phi(2^{-J}(x-u))\, du.  
\end{equation}
Additionally, the windowed scattering transform is the set of functions
\begin{equation}
S_J[\mathcal{P}_J]f = \{S_J[p]f\}_{p \in \mathcal{P}_J}.
\end{equation}
This operator is similar to a convolution neural network because along each path (analogous to each layer of a convolutional neural network) a convolution, a nonlinearity is applied, and feature aggregation occurs via the low pass filter. The scattering norm for any set of paths $\Omega$ is 
\begin{equation}
\|S_J[\Omega]f\|^2 = \sum_{p \in \Omega} \|S_J[p]f\|^2_2.
\end{equation}

Under very stringent conditions on the mother wavelet, the author of \cite{mallat2012group} was able to prove an isometry property for the windowed scattering transform. However, the problem with the admissibility condition in \cite{mallat2012group} is that there are very few classes of wavelets that are admissible. The author of \cite{mallat2012group} mentions an analytic cubic spline Battle-Lemari{\'e} wavelet is admissible in one dimension, but provides no other examples. On a related note, \cite{waldspurger2017exponential} has shown that scattering coefficients have exponential decay for $n = 1$ under relatively mild assumptions, but her proof only applies for $n = 1$, which makes the admissibility condition still necessary for $n \geq 2$. Additionally, to our knowledge, there are no examples in the literature of wavelets that satisfy the admissibility condition when $n > 1$.  

The windowed scattering transform has three important properties that are helpful for certain machine learning tasks. The first two are the following:
\begin{enumerate}
    \item The windowed scattering transform is a well-defined mapping on $\Lb^2(\mathbb{R}^n)$ and nonexpansive. In particular, for all $f, h \in \mathbf{L}^2(\mathbb{R}^n)$,
    \begin{equation}\|S_J[P_J]f - S_J[P_J]h\| \leq \|f-h\|_2.
    \end{equation}
    \item Let the translation of a function be denoted as $L_cf(u) = f(u-c).$ For certain classes of wavelets, we have
    \begin{equation}
    \lim_{J \to \infty} \|S_J[P_J]f - S_J[P_J]L_cf\| = 0   
    \end{equation}
    for all $c \in \mathbb{R}^n$ and for all $f \in \Lb^2(\mathbb{R}^n)$. One can think of this as local translation invariance.
\end{enumerate}

Finally, for the last property, the following definition was used in \cite{mallat2012group} for Lipschitz continuity to the action of $C^2$ diffeomorphisms. Let $\mathcal{H}$ be a Hilbert space, $\tau \in C^2$, and define the operator $L_\tau f(x) = f(x -\tau (x))$. A translation invariant operator $\Phi$ is said to be Lipschitz continuous to the action of $C^2$ diffeomorphisms if for any compact $\Omega \subset \mathbb{R}^n$, there exists $C_\Omega$ such that for all $f \in \Lb^2(\mathbb{R}^n)$ supported in $\Omega$ and all $\tau \in C^2(\mathbb{R}^n)$, we have
\begin{equation} \label{diffeo stability}
\|\Phi(f) - \Phi(L_\tau f) \|_{\mathcal{H}} \leq C_\Omega\left(\|D\tau\|_\infty + \|D^2\tau\|_\infty\right)\|f\|_2.
\end{equation} The idea is that the difference in norm is proportional to the size of $\|D\tau\|_\infty + \|D^2\tau\|_\infty$, which indicates how much $L_\tau$ deforms $f$. In particular, the author of \cite{mallat2012group} show that \eqref{diffeo stability} holds for the windowed scattering transform.

The concept of stability to diffeomorphisms has become a major point of study after the publication of \cite{mallat2012group}. Based on the definition above, there has been a lot of interest in exploring the stability of various operators related to machine learning and data science. For example, \cite{geometric, gama} extend the scattering transform and stability of the scattering transform to graphs and compact Riemannian manifolds, respectively; the authors in \cite{nicola} loosen the restriction on the regularity of $\tau$. Other papers explore stability for different operators with desirable properties for machine learning \cite{wiatowski, koller, czaja, bietti}. 

Although much work has appeared in recent years about operators similar to the scattering transform and about generalizations of the scattering transform, there are still some loose ends left in \cite{mallat2012group} that have not been explored yet. First, while the author of \cite{mallat2012group} does explore creating a norm for the nonwindowed scattering transform, he does not actually prove the norm is stable to diffeomorphisms. 
We consider a less stringent definition for stability to diffeomorphisms in the same spirit as the definition in \cite{mallat2012group} for this paper. Let $V_1$ and $V_2$ be normed vector spaces. Then we say that a translation invariant operator $\Phi: V_1 \to V_2$ is said to be Lipschitz continuous to the action of $C^2$ diffeomorphisms if for any compact $\Omega \subset \mathbb{R}^n$, there exists $C_{\Omega,\tau}$ such that for all $f \in V_1$ supported in $\Omega$ and all $\tau \in C^2(\mathbb{R}^n)$, we have
\begin{equation}
\|\Phi(f) - \Phi(L_\tau f) \|_{V_2} \leq C_{\Omega, \tau}\|f\|_{V_1},
\end{equation} where $C_{\Omega, \tau} \to 0$ as $\|D\tau\|_\infty + \|D^2\tau\|_\infty \to 0.$ Like with equation \eqref{diffeo stability}, $\|\Phi(f)-\Phi(L_\tau f)\|_{V_2}$ depends on $\|D\tau\|_\infty + \|D^2\tau\|_\infty$. 

%motivate scattering transform better here
Using this definition, we consider a slightly different problem than the author of \cite{mallat2012group} did for the nonwindowed scattering transform. The scattering transform introduced in \cite{mallat2012group} was a collection of $\Lb^1 (\R^n)$ norms of various cascades of dyadic wavelet convolutions and modulus nonlinearities applied to a signal. Here, we extend the definition of the scattering transform to the continuous wavelet transform and for $\Lb^q (\R^n)$ norms with $q \in [1,2]$. For a continuous dilation parameter $\lambda \in \R_+$ we define the dilations of $\psi$ as:
\begin{equation*}
    \forall \, \lambda \in \R_+ \, , \quad \psi_{\lambda} (x) := \lambda^{-n/2} \psi (\lambda^{-1} x) \, ,
\end{equation*}
which preserves the $\Lb^2 (\R^n)$ norm of $\psi$:
\begin{equation*}
    \| \psi_{\lambda} \|_2 = \| \psi \|_2 \, , \quad \forall \, \lambda \in \R_+ \, .
\end{equation*} For the continuous wavelet transform, the one layer wavelet scattering transform with $\Lb^q (\R^n)$ norm is the function $S_{\text{cont},q} : \R_+ \rightarrow \R$ defined as:
\begin{equation}
    \forall \, \lambda \in \R_+ \, , \quad S_{\text{cont},q} f (\lambda) := \| f \ast \psi_{\lambda} \|_q \, .
\end{equation}

For a dyadic dilation parameter $j \in \Z$ we define dilations of $\psi$ as:
\begin{equation*}
    \forall \, j \in \Z \, , \quad \psi_j (x) = 2^{-nj} \psi (2^{-j} x) \, ,
\end{equation*}
which preserves the $\Lb^1 (\R^n)$ norm of $\psi$:
\begin{equation*}
    \| \psi_j \|_1 = \| \psi \|_1 \, , \quad \forall \, j \in \Z \, .
\end{equation*} The one layer wavelet scattering transform for the dyadic wavelet transform is the function $S_{\text{dyad},q} f : \Z \rightarrow \R$ defined as:
\begin{equation}
    \forall \, j \in \Z \, , \quad S_{\text{dyad},q} f(j) := \| f \ast \psi_j \|_q \, .
\end{equation}

More generally, the $m$-layer wavelet scattering transforms $S_{\text{cont},q}^m f: \R_+^m \rightarrow \R$ and $S_{\text{dyad},q}^m f: \Z^m \rightarrow \R$ are defined as
\begin{align}
    S_{\text{cont},q}^m f(\lambda_1, \ldots, \lambda_m) &:= \| ||f \ast \psi_{\lambda_1}| \ast \psi_{\lambda_2}| \ast \cdots |\ast \psi_{\lambda_m} \|_q \, , \\
    S_{\text{dyad},q}^m f(j_1, \ldots, j_m) &:= \| ||f \ast \psi_{j_1}| \ast \psi_{j_2}| \ast \cdots |\ast \psi_{j_m} \|_q \, .
\end{align}
 
This is similar to working with a windowed scattering transform with a finite number of layers. However, our operator is different from the operator $S_J$ in \cite{mallat2012group} because it does not contain the filter $A_J$ to aggregate low frequency information, so the scale parameter in our formulation is not bounded above or below. Additionally, because the averaging filter is replaced $\Lb^q (\R^n)$ norms, our representation is fully translation invariant rather than translation invariant as $J \to \infty$.

As for the significance of using $\Lb^q (\R^n)$ norms to replace the averaging filter, there is one area with direct application: quantum energy regression tasks \cite{QMWavelet1}, where a representation that is similar to the rotation invariant representation in Section 6.2 has already been used for quantum energy regression. 

Given a configuration of atoms, we would like to estimate the ground state energy of the configuration. Suppose we have a molecule with $K$ atoms with nuclear charges $z_k$ and nuclear positions $p_k$ with $k = 1, \ldots, K$.  The state $x$ of a molecule is given by
\begin{equation}
x = \{(p_k, z_k) \in \mathbb{R}^3 \times \mathbb{R}\, : \, k = 1 \ldots, K\},
\end{equation} Due to how we have defined our state, we would like our representation to have the following properties:
\begin{itemize}
    \item \textbf{Permutation Invariance}: the energy should not depend on the index of the molecules.
    \item \textbf{Deformation Stability}: small deformations of the molecule should only lead to small changes in energy of the system.
    \item \textbf{Isometry Invariance}: the energy should be invariant to group actions such as translations, rotations, and other general isometries.
    \item \textbf{Multiscale Interactions}: molecules have many interactions terms, and these interaction terms depend on the pairwise distance between atoms (i.e. short range covalent bonds and longer range Van Der Waals interactions). 
\end{itemize}

The rotation invariant version of our scattering transform in section 6 satisfies permutation invariance, deformation stability, and has multiscale interactions based on the proofs we've provided. We do not prove isometry invariance, but the operator is rotation and translation invariant.

Motivated by DFT theory, the paper \cite{QMWavelet1} uses a dictionary of one and two layer scattering norms with $q = 1$ and $q = 2$ to get (at the time) state-of-the-art results for energy regression tasks for planar molecules. In particular, scattering operators with $q = 1$ scaled with the number of atoms in the system and $q =2$ encoded pairwise interactions. The motivation for using $1 < q < 2$ comes from \cite{QMWavelet2, QMWavelet3}, which based on the Thomas–Fermi–Dirac–von Weizs{\"a}cker model \cite{CANCES20033}, also use scattering norms with $q = 4/3, 5/3$. Later papers, like \cite{QMWavelet2, QMWavelet3}, use a similar representation, involving spherical harmonics, for 3D quantum energy regression.

Generalizing to stochastic processes, one can also consider scattering moments \cite{mallat2012group, bruna2013audio}, which have similar desirable properties as the nonwindowed scattering transform. Applications include, but are not limited to, audio texture synthesis \cite{bruna2013audio} and cosmology \cite{allys2019rwst}. The main idea in all these applications is that the nonwindowed scattering transform has desirable mathematical properties and provides a small number of relevant descriptors for high dimensional, complicated data.

\begin{remark}
We can replace all the modulus operators with any contraction mapping (or use different contraction mappings in each layer) in the definition above, and all the proofs in the rest of this paper will still work. In particular, the modulus can be replaced with a complex version of the rectified linear unit (ReLU) nonlinearity, $\max(0, \text{Re}(a_{i}))_{i = 1, \ldots, n}$ for $a \in \mathbb{C}^n$, which is a popular choice for complex neural networks. Nonetheless, we will use the modulus operator throughout this paper without any loss of generality.
\end{remark}

We provide a general roadmap for this paper. Section 2 will cover notation, basic properties about wavelets and the wavelet scattering operator, and harmonic analysis that will be necessary for the paper. In Section 3, we provide norms for an $m$-layer wavelet scattering transforms and prove that the operators are well defined mappings into specific spaces when $1 \leq q \leq 2$. For Section 4, we explore conditions under which the $m$-layer scattering transform is stable to dilations, and we generalize our results to diffeomorphisms in Section 5. Lastly, in Section 6, we formulate two new translation invariant operators that are stable to diffeomorphisms. The first is rotation equivariant, and the second is rotation invariant. Our contributions include, but are not limited to, the following:
\begin{itemize}
    \item We formulate an extension of the dyadic wavelet scattering operator for a finite, arbitrary number of layers with parameter $q \in [1,2]$ by applying $\Lb^q(\mathbb{R}^n)$ norms instead of $\Lb^1(\mathbb{R}^n)$ norms. Additionally, we formulate a wavelet scattering operator with $q \in [1,2]$ that uses a continuous scale parameter, like the continuous wavelet transform.
    \item We create a new finite depth scattering norm using dyadic and continuous scales in the case when $q \in [1,2]$, and prove that the mappings are well defined and provide theoretical justification for a broader class of wavelets that make the scattering transform Lipchitz continuous to the action of $C^2$ diffeomorphisms. However, the trade-off is that our stability bound depends on the number of layers.
    \item We provide a condition for norm equivalence in the case of $q = 2$ that is less stringent.
    \item In the case of $q \in( 1,2]$, we prove that our norm is stable to diffeomorphisms $\tau \in C^2(\mathbb{R}^n)$ provided that $\|\tau\|_\infty < \tfrac{1}{2n}$ and the wavelet and its first and second partial derivatives have sufficient decay. In the case of $q = 1$, we show stability to dilations.
    \item We extend our formulation to include invariance or equivariance to the action of rotations $R \in \text{SO}(n)$.
\end{itemize}

\section{Notation and Basic Properties}
\label{sec: notation and basic properties}
We start by providing basic notation that we will use in this paper and proceed to give basic definitions and properties that will be necessary for our results.

\subsection{Function Spaces}
Set $\R_+$ to be the positive real numbers, i.e. $\R_+ := (0, \infty).$ The gradient of a function  $f : \R^n \rightarrow \C$ is given by $\nabla f$, the Jacobian of a function $f: \R^n \to \R^m$ is given by $Df$, and the Hessian is given by $D^2f$. For $1\leq q < \infty$, the $\Lb^q (\R^n)$ norm of a function $f : \R^n \rightarrow \C$ is $\| f \|_q := \left[ \int_{\R^n} |f(x)|^q \, dx\right]^{1/q}.$ When $q = \infty$, $\|f\|_\infty := \text{ess sup} |f|.$ We will also use the notation, $\|\Delta f\|_\infty = \sup_{x, y \in \mathbb{R}^d} |f(x)-f(y)|$, which should not be mistaken for applying a Laplacian operator. Greek letters with a vector symbol, such as $\vec{\alpha} = (\alpha_1, \cdots, \alpha_n)$, will be a multi-index of nonnegative integers; additionally, we write $|\vec{\alpha}| = \alpha_1+ \cdots + \alpha_n$, and the usage will be clear from context. The operator $D^{\vec{\alpha}}$ is a multi-index of derivatives: $D^{\vec{\alpha}} f =  \frac{ \partial^{|\vec{\alpha}|} }{\partial x_1^{\alpha_1} \cdots \partial x_n^{\alpha_n} } f$. For integer $s\geq 0$, we define the function space
$\mathbf{H}^s(\mathbb{R}^n) = \{ f \in \mathbf{L}^2(\mathbb{R}^n): D^{\vec{\alpha}}f \in \mathbf{L}^2(\mathbb{R}^n) \text{ for } |\vec{\alpha}| \leq s \}.$ 

The Fourier transform of a function $f \in \Lb^1 (\R^n)$ is the function $\hf \in \Lb^{\infty} (\R^n)$ defined as:
\begin{equation*}
    \forall \, \omega \in \R^n \, , \quad \hf (\omega) := \int_{\R^n} f(x) e^{-i x \cdot \omega} \, dx \, .
\end{equation*} 
The Hilbert transform of a function $f \in \mathbf{L}^1(\mathbb{R})$ is denoted by $Hf$ and is defined as:
\begin{equation*}
    Hf(x) := \lim_{\epsilon \rightarrow 0} \int_{|x - y| > \epsilon} \frac{f(y)}{x-y} \, dy \, .
\end{equation*}
The map $H$ is a convolution operator in which $f$ is convolved against the function $1/x$. We note that 
\begin{equation*}
    H : \Lb^q (\R) \rightarrow \Lb^q (\R) \, , \quad \forall \, 1 < q < \infty \, ,
\end{equation*}
however the result is not true for $q = 1$, i.e., if $f \in \Lb^1 (\R)$ it is not necessarily true that $Hf \in \Lb^1 (\R)$. We thus introduce the Hardy space. We denote the Hardy space as $\Hb^1 (\R)$ and it consists of those functions $f \in \Lb^1 (\R)$ such that $Hf \in \Lb^1(\R)$ as well. For $f \in \Hb^1 (\R)$ the Hardy space norm is $\| f \|_{\Hb^1 (\R)}$, which we define as (see Corollary 2.4.7 of \cite{grafakosmodern})
\begin{equation} \label{eqn: hardy space norm}
    \| f \|_{\Hb^1 (\R)} := \| f \|_1 + \| H f \|_1 \, .
\end{equation}
One can show that if $f \in \Hb^1 (\R)$, then $f$ must necessarily have zero average. An important property of the Hilbert transform and convolution is the following:
\begin{equation*}
    H(f \ast g) = Hf \ast g = f \ast Hg \, , \quad f \in \Lb^p (\R) \, , \, g \in \Lb^q (\R) \, , \quad 1 < \frac{1}{p} + \frac{1}{q} \, .
\end{equation*}
We have a similar definition for Hardy spaces when $n \geq 2$. For $1 \leq j \leq n$, define the $j^{\text{th}}$ Riesz transform as
\begin{equation} \label{eqn: hardy space norm riesz}
    R_jf(x) = \lim_{\varepsilon \to 0} \int_{|x-y|>\varepsilon} \frac{x_j-y_j}{|x-y|^{n+1}} f(y) \, dy \, ,
\end{equation}
where $x = (x_1, \ldots, x_n)$ and $y = (y_1, \ldots, y_n)$. The Hardy space $f \in \Hb^1 (\R^n)$ consists of functions $f$ such that $f \in \Lb^1 (\R^n)$ and $R_jf \in \Lb^1(\R^n)$ for $1 \leq j \leq n$ as well.
For $f \in \Hb^1 (\R^n)$ the Hardy space norm is $\| f \|_{\Hb^1 (\R^n)}$, which we define as (see Corollary 2.4.7 of \cite{grafakosmodern})
\begin{equation} \label{eqn: hardy space norm Riesz}
    \| f \|_{\Hb^1 (\R^n)} := \| f \|_1 + \sum_{j = 1}^n\| R_j f \|_1 \, .
\end{equation} 

\subsection{Wavelets}
We let $\psi \in \Lb^1 (\R^n) \cap \Lb^2 (\R^n)$ be a wavelet, which means it is a function that is localized in both space and frequency and has zero average, i.e.,
\begin{equation*}
    \int_{\R^n} \psi (x) \, du = 0 \, .
\end{equation*}

Assume $f \in \Lb^2(\mathbb{R}^n)$. The continuous wavelet transform $\Wc f \in \Lb^2 (\R^n \times \R_+)$ is defined as:
\begin{equation*}
    \forall \, (x, \lambda) \in \R^n \times \R_+ \, , \quad \Wc f(x,\lambda) := f \ast \psi_{\lambda} (x) \, .
\end{equation*}
Furthermore, if $\psi$ satisfies the following admissibility condition
\begin{equation} \label{def: continuous littlewood paley admissible}
    \int_0^\infty \frac{|\hpsi (\lambda \omega)|^2}{\lambda} d \lambda = \mathcal{C}_{\psi} \, , \quad \forall \, \omega \in \mathbb{R}^n \setminus\{0\} \, ,
\end{equation} 
for some $\mathcal{C}_{\psi} > 0$, then we will say that $\psi$ is a Littlewood-Paley wavelet for the continuous wavelet transform. If $\psi$  satisfies (\ref{def: continuous littlewood paley admissible}), one can show that the norm $\Wc f$ computed with a weighted measure $(dx, d\lambda / \lambda^{n + 1})$ on $\R^n \times \R_+$ is well defined:
\begin{equation*}
    \| \Wc f \|_{\Lb^2 (\R^n \times \R_+)}^2 := \int_0^{\infty} \int_{\R^n} |\Wc f(x,\lambda)|^2 \, dx \, \frac{d\lambda}{\lambda^{n+1}} = \int_0^{\infty} \int_{\R^n} |f \ast \psi_{\lambda} (x)|^2 \, dx \, \frac{d\lambda}{\lambda^{n+1}} = \int_0^{\infty} \| f \ast \psi_{\lambda} \|_2^2 \, \frac{d\lambda}{\lambda^{n+1}} \, .
\end{equation*}
We note, in fact, that one can show:
\begin{equation*}
    \| \Wc f \|_{\Lb^2 (\R^n \times \R_+)}^2 = \beta \cdot \Cc_{\psi} \| f \|_2^2 \, .
\end{equation*}
where
\begin{equation} \label{eqn: beta definition}
    \beta = \left\{
    \begin{array}{ll}
        1/2 & \text{if }\psi \text{ is real valued} \\
        1 & \text{if } \psi \text{ is complex valued}
    \end{array}
    \right. \, .
\end{equation}

For a function $f \in \Lb^2 (\R^n)$ we define the dyadic wavelet transform $Wf \in \ellb^2 (\Lb^2 (\R^n))$ as
\begin{equation*}
    Wf = \left( f \ast \psi_j \right)_{j \in \Z} \,.
\end{equation*}
If $\psi$ satisfies
\begin{equation} \label{eqn: dyadic littlewood paley admissible}
    \sum_{j \in \mathbb{Z}} |\hpsi(2^{j}\omega)|^2 = \hat{C}_\psi, \quad \forall \omega \in \mathbb{R}^n \setminus\{0\} \, ,
\end{equation}
for some $\hat{C}_\psi > 0$, then we will say that $\psi$ is a Littlewood-Paley wavelet for the dyadic wavelet transform. If $\psi$  satisfies (\ref{eqn: dyadic littlewood paley admissible}), one can show that the norm $W f$ given below is well defined:
\begin{equation*}
    \| Wf \|_{\ellb^2 (\Lb^2 (\R))}^2 := \sum_{j \in \Z} \| f \ast \psi_j \|_2^2 \, .
\end{equation*}
In fact, we have the following norm equivalence:
\begin{equation*}
    \| Wf \|_{\ellb^2 (\Lb^2 (\R))}^2 = \beta \cdot \hat{C}_{\psi} \| f \|_2^2 \, ,
\end{equation*}
where $\beta$ is defined in \eqref{eqn: beta definition}.

\subsection{Operator Valued Spaces}
Consider a Banach space $\mathcal{B}$. Suppose $f: \mathbb{R}^n \to \mathcal{B}$ and $x \to \|f(x)\|_\mathcal{B}$ is measurable in the Lebesgue sense. Define $\Lb^p_\mathcal{B}(\mathbb{R}^n)$ for $1 \leq p < \infty$ to be
\begin{equation*}
    \|f\|_{\Lb^p_\mathcal{B}(\mathbb{R}^n)}^p =  \int_{\mathbb{R}^n}\|f(x)\|_\mathcal{B}^p \, dx \, .
\end{equation*}
Also, for $1 \leq p < \infty$, define
\begin{equation*}
    \|f\|_{\Lb_\mathcal{B}^{p,\infty}(\mathbb{R}^n)} = \sup_{\delta > 0} \delta \cdot m(\{x \in \mathbb{R}^n : \|f(x)\|_{\mathcal{B}} > \delta\})^{1/p} \, .
\end{equation*}
We also have the following relation:
\begin{equation*}
    \|f\|_{\Lb_\mathcal{B}^{p,\infty}(\mathbb{R}^n)} \leq \|f\|_{\Lb^p_{\mathcal{B}}(\mathbb{R}^n)} \, .
\end{equation*}
Note that for $f: \mathbb{R}^n \to \mathbb{R}^n$, 
\begin{equation*}
    \|f\|_{\Lb^p_{\mathbb{R}^n}(\mathbb{R}^n)}^p =  \int_{\mathbb{R}^n}\|f(x)\|_{\mathbb{R}^n}^p \, dx = \int_{\mathbb{R}^n}|f(x)|^p \, dx = \|f\|_p^p \, .
\end{equation*}

\section{Wavelet Scattering is a Bounded Operator}
In this section we explore for which $q > 0$ and $m \geq 1$ the wavelet scattering transforms $S_{\text{cont},q}^m f$ and $S_{\text{dyad},q}^m f$ are well-defined as functions in some Banach space (i.e., have finite norm), and under what circumstances. 

Let $\psi$ be a wavelet. We assume that $\psi$ has the following properties:
\begin{equation}\label{eqn: decay condition}
|\psi(x)| \leq A(1+|x|)^{-n-\varepsilon} 
\end{equation}
\begin{equation} \label{eqn: holder condition}
\int_{\mathbb{R}^n} |\psi(x-y)-\psi(x)|\, dx \leq A|y|^{\varepsilon'} \, ,
\end{equation}
for some constants $A, \varepsilon', \varepsilon > 0$ and for all $h \neq 0$. 

Consider the Littlewood-Paley $G$-function
\begin{equation}
    G_\psi(f)(x) = \left(\int_{(0, \infty)}|f \ast t^{-n}\psi(x/t)|^2 \frac{dt}{t}\right)^{1/2} \, .
\end{equation}
Let $\mathcal{B} = \Lb^2\left((0, \infty), \frac{dt}{t}\right)$. We can rewrite this as a Bochner integral by considering the function $K(x) = (t^{-n/2}\psi_t(x))_{t>0}$. This is a mapping $K:\mathbb{R}^n \to \mathcal{B}$ and the function $x \to \|K(x)\|_\mathcal{B}$ is measurable. Also, if we let
\begin{equation*}
    \mathcal{T}(f)(x) = \left(\int_{\mathbb{R}^n}t^{-n/2}\psi_t(x-y)f(y) \, dy\right)_{t > 0} = \left((t^{-n/2}\psi_t \ast f)(x)\right)_{t > 0} \, ,
\end{equation*}
we observe that
\begin{equation*}
    G_\psi(f)(x) = \|\mathcal{T}(f)(x)\|_\mathcal{B}
\end{equation*}
and
\begin{equation*}
    \|G_\psi(f)\|_p^p = \|\mathcal{T}(f)\|_{L^p_\mathcal{B}(\mathbb{R}^n)}^p \, .
\end{equation*}
From Problem 6.1.4 of \cite{grafakos}, the two properties above for the wavelet $\psi$ imply that 
\begin{equation} \label{eqn: bochner decay condition}
    \|K(x)\|_\mathcal{B} \leq \frac{c_n A}{|x|^n} \, ,
\end{equation}
and 
\begin{equation} \label{eqn: bochner holder condition}
\sup_{y \in \mathbb{R}^n \setminus \{0 \}}\int_{|x| \geq 2|y|} \|K(x - y) - K(x)\|_\mathcal{B} dx \leq c_n'A \, ,
\end{equation}
where $c_n$ and $c_n'$ depend only on $n$, $\varepsilon$, and $\varepsilon'$. We will omit the dependence on $\varepsilon$ and $\varepsilon'$ throughout the rest of this paper, and this will have no effect on any of our proofs.

\begin{remark}
For the rest of this paper, we will write $G$ in place of $G_\psi$ when referring to the $G$-function because the dependence on the mother wavelet is clear.
\end{remark}

\begin{remark}
Note that \eqref{eqn: holder condition} holds under the alternative condition
\begin{equation}
|\nabla \psi(x)| \leq A(1+ |x|)^{-n-1-\epsilon'}.
\end{equation} This is a consequence of Mean Value Theorem.
\end{remark}

We have the following result taken from Problem 6.1.4 of \cite{grafakos} and from Chapter V of \cite{vectorvaluedinequalities}.

\begin{lemma}[\cite{grafakos, vectorvaluedinequalities}] \label{thm: LP g-function Lp bounded}
Assume that $\psi$ is defined as above and satisfies (\ref{eqn: bochner decay condition}) and (\ref{eqn: bochner holder condition}). Then the operator $G$ is bounded from $\Lb^2(\mathbb{R}^n)$ to $\Lb^2(\mathbb{R}^n)$. Also, for $p \in (1, \infty)$ and $\mathcal{B} = \Lb^2(\mathbb{R}_+, dt/t)$, we have
\begin{equation}
    \|\mathcal{T}f\|_{\Lb^p_\mathcal{B}(\mathbb{R}^n)} \leq C_n A \max(p, (p-1)^{-1}) \|f\|_{\Lb^p(\mathbb{R}^n)} \, ,
\end{equation}
for some $C_n$. For all $f \in \Lb^1(\mathbb{R}^n)$, we also have
\begin{equation}
    \|\mathcal{T}f\|_{\Lb_\mathcal{B}^{1,\infty}(\mathbb{R}^n)}\leq  C_n' A \|f\|_{\Lb^1(\mathbb{R}^n)}
\end{equation}
and
\begin{equation}
    \|\mathcal{T}f\|_{\Lb_\mathcal{B}^{1}(\mathbb{R}^n)} \leq  C_n' A \|f\|_{\Hb^1(\mathbb{R}^n)} \, ,
\end{equation}
for some $C_n'$. 
\end{lemma} 

\begin{remark}
We can also formulate similar bounds for the Littlewood-Paley $\gk$ operator 
\begin{equation}
	\gk (f)(x) := \left[ \sum_{j\in\Z} |\psi_{j} \ast f(x)|^2 \right]^{1/2}
\end{equation} 
using similar arguments.
\end{remark}

\begin{remark}
Let $\psi$ be a wavelet that has properties \eqref{eqn: decay condition} and \eqref{eqn: holder condition}. Then with the $\Lb^2$ normalized dilations, the Littlewood-Paley $G$-function can be written as:
\begin{align}
    G(f)(x) &= \left[ \int_0^{\infty} |f \ast \psi_{\lambda} (x)|^2 \, \frac{d\lambda}{\lambda^{n+1}} \right]^{1/2} \, . \label{eqn: G function with continuous wavelet}
\end{align}
Note that the $\lambda$ measure for $G(f)$ matches the measure in defining the norm of $\mathcal{W} f$. 
\end{remark}

\subsection{The $\Lb^2 (\R^n)$ Wavelet Scattering Transform}
In this subsection we prove the $\Lb^2 (\R^n)$ scattering transforms are bounded operators. More specifically, we prove that $S_{\text{cont},2}^m : \Lb^2 (\R^n) \rightarrow \Lb^2 (\R_+^m)$, where $\Lb^2 (\R_+^m)$ has the weighted measure defined by
\begin{equation*}
    \| S_{\text{cont},2}^m f \|_{\Lb^2 (\R_+^m)}^2 := \int_0^{\infty}\dots\int_0^{\infty} |S_{\text{cont},2}^m f(\lambda_1, \ldots,\lambda_m)|^2 \,\frac{d\lambda_1}{\lambda_1^{n+1}} \dots \frac{d\lambda_m}{\lambda_m^{n+1}}
\end{equation*}
and we show that $\| S_{\text{cont},2}^m f \|_{\Lb^2 (\R_+^m)} \leq C \| f \|_{\Lb^2 (\R^n)}$. We also show that $S_{\text{dyad},2}^m : \Lb^2(\mathbb{R}^n) \rightarrow \ellb^2 (\Z^m)$, where
\begin{equation*}
    \| S_{\text{dyad},2}^m f \|_{\ellb^2 (\Z^m)}^2 := \sum_{j_m \in \Z} \ldots \sum_{j_1 \in \Z} | S_{\text{dyad},2}^m f (j_1, \ldots, j_m) |^2.
\end{equation*}

\begin{proposition} \label{prop:L1MapMultipleLayersCont}
For any wavelet satisfying \eqref{eqn: decay condition} and \eqref{eqn: holder condition}, we have $S_{\text{cont},2}^m:\Lb^2(\mathbb{R}^n) \rightarrow \Lb^2(\mathbb{R}^m_+)$ and $S_{\text{dyad},2}^m:\Lb^2(\mathbb{R}^n) \rightarrow \ell^2(\mathbb{Z}^m)$. 
\end{proposition} 
\begin{proof}
The proof of the dyadic case is essentially identical to the proof given below and is thus omitted. The case of $m = 1$ follows by an application of Fubini's Theorem:
\begin{align*}
    \| S_{\text{cont},2} f \|_{\Lb^2 (\R_+)}^2 &= \int_0^{\infty} \| f \ast \psi_{\lambda} \|_2^2 \, \frac{d\lambda}{\lambda^{n+1}} \\
    &= \int_0^{\infty} \int_{\mathbb{R}^n}|(f \ast \psi_{\lambda})(x)|^2 \, dx \, \frac{d\lambda}{\lambda^{n+1}} \\
    &=  \int_{\mathbb{R}^n} |G(f)(x)|^2 \, dx \\
    & \leq C \|f\|^2_2
\end{align*} by boundedness of the G-function. 
Now we proceed by using induction. Assume that we have $\| S_{\text{cont},2}^m f \|_{\Lb^2 (\R_+^m)}^2 \leq C_m \|f\|^2_2$. Let $\Wc_{t}f = f \ast \psi_{t}$, define $Mf = |f|$, and $U_\lambda = MW_\lambda$ for notational brevity. Then notice that 
$$\||||f*\psi_{\lambda_1}|*\psi_{\lambda_2}|*\dots*\psi_{\lambda_{m}}| \ast \psi_{\lambda_{m+1}}\|_2^2 = \|\Wc_{\lambda_{m+1}}U_{\lambda_m} \cdots U_{\lambda_1} f\|_2^2.$$
Substituting yields
\begin{align*}
\| S_{\text{cont},2}^{m+1} f \|_{\Lb^2 (\R_+^{m+1})} &= \int_0^{\infty}\dots\int_0^{\infty} \|\Wc_{\lambda_{m+1}}U_{\lambda_m} \cdots U_{\lambda_1} f\|_2^2 \,\frac{d\lambda_1}{\lambda_1^{n+1}} \dots \frac{d\lambda_{m+1}}{\lambda_{m+1}^{n+1}}\\
&=\int_0^{\infty} \dots\int_0^{\infty}\int_0^{\infty} \|(U_{\lambda_m} \cdots U_{\lambda_1} f) \ast \psi_{\lambda_{m+1}}\|_2^2 \frac{d\lambda_{m+1}}{\lambda_{m+1}^{n+1}} \, \frac{d\lambda_1}{\lambda_1^{n+1}} \dots\frac{d\lambda_{m}}{\lambda_{m}^{n+1}} \\
&= \int_0^{\infty}\dots\int_0^{\infty}\|U_{\lambda_m} \cdots U_{\lambda_1} f\|_{\Lb^2 (\R_+)}^2 \frac{d\lambda_1}{\lambda_1^{n+1}} \dots\frac{d\lambda_{m}}{\lambda_{m}^{n+1}}\\
& \leq C \int_0^{\infty}\dots\int_0^{\infty} \|U_{\lambda_m} \cdots U_{\lambda_1} f\|_2^2 \frac{d\lambda_1}{\lambda_1^{n+1}} \dots\frac{d\lambda_{m}}{\lambda_{m}^{n+1}} \\
& = C \int_0^{\infty}\dots\int_0^{\infty} |S_{\text{cont},2}^m(\lambda_1, \ldots, \lambda_m)|^2 \frac{d\lambda_1}{\lambda_1^{n+1}} \dots\frac{d\lambda_{m}}{\lambda_{m}^{n+1}} \\
& \leq C^{m+1} \|f\|_2^2,
\end{align*} where we used the induction hypothesis in the last line. This completes the proof.
\end{proof}
\begin{proposition} 
\label{prop:L1MapMultipleLayersInversion}
Suppose $\psi$ is a Littlewood-Paley wavelet satisfying \eqref{eqn: decay condition} and \eqref{eqn: holder condition}. Then  $S_{\text{cont},2}^m f:\Lb^2(\mathbb{R}^n) \rightarrow \Lb^2(\mathbb{R}^m_+)$ and specifically $\norm{S_{\text{cont},2}^m f}_1 = C_\psi^m\norm{f}_2^2$. Also, $S_{\text{dyad},2}^m:\Lb^2(\mathbb{R}^n) \rightarrow \ell^2(\mathbb{Z}^m)$ and $\norm{S_{\text{dyad},2}^mf}_1 = \hat{C}_\psi^m\norm{f}_2^2$.
 \end{proposition}

\begin{proof}
	We only provide the proof of the continuous case again. First consider the case $m=1$. We have:
	\begin{align*}
	\norm{S_{\text{cont},2} f}_{\Lb^2(\mathbb{R}_+)}^2 &=\int_0^{\infty} \| f \ast \psi_{\lambda} \|_2^2 \, \frac{d\lambda}{\lambda^{n+1}}  \\
	&= \frac{1}{(2\pi)^n}\int_0^{\infty} \| \hat{f} \cdot \hat{\psi}_{\lambda} \|_2^2 \, \frac{d\lambda}{\lambda^{n+1}}  \\
	&=  \frac{1}{(2\pi)^n}\int_0^{\infty} \left(\int_{\mathbb{R}^n} |\hat{f}(\omega)|^2|\hat{\psi}_\lambda(\omega)|^2\ d\omega\right) \frac{d\lambda}{\lambda^{n+1}} \\
	&= \frac{1}{(2\pi)^n}\int_{\mathbb{R}^n} \left(\int_0^{\infty}  |\hat{\psi}(\lambda\omega)|^2 \frac{d\lambda}{\lambda}\right)|\hat{f}(\omega)|^2 \ d\omega \\
	&= \frac{1}{(2\pi)^n}\int_{\mathbb{R}^n} \left( C_\psi|\hat{f}(\omega)|^2 \right) \ d\omega \\
	&= \frac{1}{(2\pi)^n}C_\psi\norm{\hat{f}}_2^2 \\
	&= C_\psi\norm{f}_2^2.
	\end{align*}
	Thus the claim holds for $m=1$. Now assume that it holds through $m$. Then by the inductive hypothesis,
	\begin{align*}
	\norm{S_{\text{cont},2}^{m}f}_{\Lb^2(\mathbb{R}_+)}^2 &= \int_0^{\infty}\dots\int_0^{\infty} \|||f*\psi_{\lambda_1}|*\psi_{\lambda_2}|*\dots*\psi_{\lambda_{m}}\|_2^2 \,\frac{d\lambda_1}{\lambda_1^{n+1}} \dots \frac{d\lambda_{m}}{\lambda_{m}^{n+1}} = C_\psi^{m}\norm{f}_2^2.
	\end{align*}
	Now consider the case of $m+1$. Similar to the previous proposition, we have  
	\begin{align*}
	\norm{S_{\text{cont},2}^{m+1}f}_{\Lb^2(\mathbb{R}_+)}^2&=\int_0^{\infty} \dots\int_0^{\infty}\left(\int_0^{\infty} \|(U_{\lambda_m} \cdots U_{\lambda_1} f) \ast \psi_{\lambda_{m+1}}\|_2^2 \frac{d\lambda_{m+1}}{\lambda_{m+1}^{n+1}}\right) \, \frac{d\lambda_1}{\lambda_1^{n+1}} \dots\frac{d\lambda_{m}}{\lambda_{m}^{n+1}} \\
	&=   C_\psi\int_0^{\infty} \dots\int_0^{\infty}|S_{\text{cont},2}^m f(\lambda_1, \ldots,\lambda_m)|^2 \frac{d\lambda_1}{\lambda_1^{n+1}} \dots\frac{d\lambda_{m}}{\lambda_{m}^{n+1}}\\
	&=  C_\psi \norm{S_{\text{cont},2}^{m}f}_{\Lb^2(\mathbb{R}_+)}^2\\
	&= C_\psi^{m+1}\norm{f}_2^2.
	\end{align*}
Thus, the claim is proven by induction.
\end{proof}

\subsection{The $\Lb^1 (\R^n)$ Wavelet Scattering Transform}
Define the notation $\Wc_{t}f = f \ast \psi_{t}$, $Mf = |f|$, and $U_t = M\Wc_t$. We now try to prove that for $m \in \mathbb{N}$, $S_{\text{cont},1}^m : \Hb^1 (\R^n) \rightarrow \Lb^2 (\R_+^m)$. The norm for $S_{\text{cont},1}^m f$ is:
\begin{align*}
    \| S_{\text{cont},1}^m f \|_{\Lb^2 (\R_+^m)} &:= \left(\int_0^{\infty}\int_0^{\infty}\cdots \int_0^{\infty} | S_{\text{cont},1}^m f(\lambda_1, \lambda_2, \ldots, \lambda_m)|^2 \, \frac{d\lambda_1}{\lambda_1^{n+1}} \, \frac{d\lambda_2}{\lambda_2^{n+1}} \, \cdots \, \frac{d\lambda_m}{\lambda_m^{n+1}}\right)^{1/2}\\
    &= \left(\int_0^{\infty}\int_0^{\infty}\cdots \int_0^{\infty} \left\| (U_{\lambda_{m-1}} \cdots U_{\lambda_1} f) \ast \psi_{\lambda_m} \right\|_1^2 \, \frac{d\lambda_1}{\lambda_1^{n+1}} \, \frac{d\lambda_2}{\lambda_2^{n+1}} \,  \cdots \, \frac{d\lambda_m}{\lambda_m^{n+1}}\right)^{1/2} \,.
\end{align*} An analogous result will also hold for the operator $\Hb^1 (\R^n) \rightarrow \ell^2 (\Z_+^m)$ with norm
\begin{equation*}
    \| S_{\text{dyad},1}^m f \|_{\ellb^2 (\Z^m)} := \left(\sum_{j_m \in \Z} \ldots \sum_{j_1 \in \Z} | S_{\text{dyad},1}^m f (j_1, \ldots, j_m) |^2\right)^{1/2}.
\end{equation*}

Before we begin, we will need an important multiplier property of the individual Riesz Transforms:
\begin{equation}
\widehat{R_jf}(\omega) = - i \frac{\omega_j}{|\omega|}\hat{f}(\omega) \, .
\end{equation} 

Let $\vec{\alpha} = (\alpha_1, \ldots, \alpha_n)$ be a multi-index with $n$-elements, and let $t = (t_1, \ldots, t_n) \in \mathbb{R}^n$. We say that $\psi$ has $k$ vanishing moments if for all $|\vec{\alpha}| < k$, we have
\begin{equation}
\int_{\mathbb{R}^n} \left(\Pi_{i=1}^n t_i^{\alpha_i}\right) \psi(t) dt = 0.
\end{equation}
The following lemmas will be necessary.

\begin{lemma} [\cite{ward2013decay}] \label{lem: riesz transform swap with derivative} Suppose that $\psi$ has $N$ vanishing moments, let $M > 1$ be an integer, let $\vec{\alpha}$ be defined as before, and let $\vec{\beta} = (\beta_1, \ldots, \beta_n)$ be a multi-index. Assume that $\psi$ satisfies the following properties:
\begin{itemize}
    \item $\psi\in \mathbf{H}^s(\mathbb{R}^d) \cap C(\mathbb{R}^d)$ for some  $s> M + \frac{n}{2}$.
    \item There exists $A > 0$ and $\epsilon \in [0,1)$ such that $\psi$ satisfies
    $$|D^{\vec{\alpha}}\psi| \leq A(1+|x|)^{-n-N-|\vec{\alpha}| + \varepsilon} \text { for } 0 \leq |\vec{\alpha}| \leq M.$$
    \item For $0 \leq |\vec{\alpha}| \leq M- 1$ and $|\vec{\beta}| < N + |\vec{\alpha}|$,
    $$\int_{\mathbb{R}^n} \Pi_{i=1}^n t_i^{\beta_i} {D}^{\vec{\alpha}} \psi(t) \, dt = 0.$$
\end{itemize} Then
$$|{D}^{\vec{\alpha}} R_i \psi(x)| = |R_i D^{\vec{\alpha}} \psi(x)| \leq A(1+|x|)^{-n-N-|\vec{\alpha}| + \varepsilon + \delta}$$
for some $0 < \delta < 1 - \varepsilon$ and $D^{\vec{\alpha}} R_i \psi$ has vanishing moments up to degree $N-1 + |\vec{\alpha}|.$
\end{lemma}

An immediate consequence is the following Lemma, which we will provide without proof. 

\begin{lemma} \label{lem: riesz transform decay}
Suppose that $\psi$ satisfies the following conditions:
\begin{itemize}
    \item $\psi\in \mathbf{H}^s(\mathbb{R}^d) \cap C(\mathbb{R}^d)$ for some  $s > 2 + \frac{n}{2}$.
    \item There exists $A > 0$ and $\epsilon \in [0,1)$ such that $\psi$ satisfies
    $$|D^{\vec{\alpha}}\psi| \leq A(1+|x|)^{-n-2-|\vec{\alpha}| + \varepsilon} \text { for } 0 \leq |\vec{\alpha}| \leq 3.$$
    \item For $0 \leq |\vec{\alpha}| \leq 2$ and $|\vec{\beta}| < 2 + |\vec{\alpha}|$,
    $$\int_{\mathbb{R}^n} \Pi_{i=1}^n t_i^{\beta_i} {D}^{\vec{\alpha}} \psi(t) \, dt = 0.$$
\end{itemize} Then $R_j \psi$ and all of its first and second partial derivatives have $O((1+ |x|)^{-n-1 + \eta})$ decay for some $\eta \in (0, 1)$.
\end{lemma} The first implication to take note of is that $R_j \psi$ is a wavelet with "good" decay of itself and all its first and second partial derivatives. Note that the strict decay on the partial derivatives is necessary for technical reasons in later proofs, but decay on all second partial derivatives can be relaxed for the following theorem. 

\begin{theorem} \label{thm: m-layer L1 bounded norm}
Let $\psi$ be a wavelet satisfying Lemma \ref{lem: riesz transform decay} and let $S_{\text{cont},1}^m$ be defined as above. Then for $f \in \Hb^1 (\R^n)$, there exists a constant $C_m$ such that
\begin{equation*}
    \| S_{\text{cont},1}^m f \|_{\Lb^2 (\R_+^m)} \leq C_m \|f\|_{\Hb^1 (\R^n)} \, .
\end{equation*}
Additionally, 
?$$\| S_{\text{dyad},1}^m f \|_{\ellb^2 (\Z^m)} \leq C_m \| f \|_{\Hb^1 (\R^n)}.$$
\end{theorem}

\begin{proof} We proceed by induction and only provide a proof for the continuous case because the dyadic case follows by almost identical reasoning. Let $f \in \Hb^1 (\R^n)$ throughout the proof. By Minkowski's integral inequality (\cite{hardy1988inequalities}, Theorem 202), we have
\begin{align*}
	\| S_{\text{cont},1}f \|_{\Lb^2(\R_+)} &= \left(\int_0^\infty \| f \ast \psi_\lambda \|_1^2 \, \frac{d\lambda}{\lambda^{n+1}}\right)^{1/2} \\
	&= \left(\int_0^\infty \left(\int_{\R^n} |f \ast \psi_\lambda(x)| \ dx \right)^2 \, \frac{d\lambda}{\lambda^{n+1}}\right)^{1/2} \\
	&\leq\left( \int_{\R^n} \left( \int_0^\infty |f \ast \psi_\lambda(x)|^2 \, \frac{d\lambda}{\lambda^{n+1}} \right)^{1/2} \, dx \right) \\
	&= \int_{\R^n} G(f)(x) \, dx  \\
	&= \| G(f) \|_1 \\
	&\leq C \| f \|_{\Hb^1 (\R^n)} \, ,
\end{align*}
where in the last inequality we used Lemma \ref{thm: LP g-function Lp bounded}. 

Now we assume that there exists some $m \geq 1$ such that  
$$\| S_{\text{cont},1}^m f \|_{\Lb^2 (\R_+^m)} \leq C_m \|f\|_{\Hb^1 (\R^n)}.$$
We have
\begin{align*}
     \| S_{\text{cont},1}^{m+1} f \|_{\Lb^2 (\R_+^{m+1})}  &= \left(\int_0^{\infty}\cdots \int_0^{\infty} \left\| (U_{\lambda_{m}} \cdots U_{\lambda_1} f) \ast \psi_{\lambda_{m+1}} \right\|_1^2 \, \frac{d\lambda_1}{\lambda_1^{n+1}} \cdots \frac{d\lambda_{m+1}}{\lambda_{m+1}^{n+1}} \right)^{1/2}\\
     &= \left(\int_0^{\infty}\cdots \int_0^{\infty} \left(\int_{\mathbb{R}^n}\left|(U_{\lambda_{m}} \cdots U_{\lambda_1} f) \ast \psi_{\lambda_{m+1}}\right| \, dx \right)^2\, \frac{d\lambda_1}{\lambda_1^{n+1}} \cdots \frac{d\lambda_{m+1}}{\lambda_{m+1}^{n+1}}\right)^{1/2} \\
     &\leq \left(\int_0^{\infty}\cdots \int_0^{\infty} \left(\int_{\mathbb{R}^n}\left[\int_0^{\infty}\left|(U_{\lambda_{m}} \cdots U_{\lambda_1} f) \ast \psi_{\lambda_{m+1}}\right|^{2} \, \frac{d\lambda_{m+1}}{\lambda_{m+1}^{n+1}} \right]^{1/2}\,dx\right)^2\, \frac{d\lambda_1}{\lambda_1^{n+1}} \cdots \frac{d\lambda_{m}}{\lambda_{m}^{n+1}}\right)^{1/2} \\
     &= \left(\int_0^{\infty}\cdots \int_0^{\infty} \left[\int_{\mathbb{R}^n}G(U_{\lambda_{m}} \cdots U_{\lambda_1} f)(x) \, dx \right]^2\, \frac{d\lambda_1}{\lambda_1^{n+1}} \cdots \frac{d\lambda_{m}}{\lambda_{m}^{n+1}}\right)^{1/2}\\
     &= \left(\int_0^{\infty}\cdots \int_0^{\infty} \|G(U_{\lambda_{m}} \cdots U_{\lambda_1} f)\|_1^2 \frac{d\lambda_1}{\lambda_1^{n+1}} \cdots \frac{d\lambda_{m}}{\lambda_{m}^{n+1}}\right)^{1/2}\\
      &= \left(\int_0^{\infty}\cdots \int_0^{\infty} \|G(\Wc_{{\lambda_{m}}}U_{\lambda_{m-1}} \cdots U_{\lambda_1} f)\|_1^2 \frac{d\lambda_1}{\lambda_1^{n+1}} \cdots \frac{d\lambda_{m}}{\lambda_{m}^{n+1}}\right)^{1/2}
\end{align*} since the $G$ function has a modulus already. \newline 

It follows that
\begin{align*}
      \| S_{\text{cont},1}^{m} f \|_{\Lb^2 (\R_+^m)}&\leq C \left(\int_0^{\infty}\cdots \int_0^{\infty} \| \Wc_{{\lambda_{m}}}U_{\lambda_{m-1}} \cdots U_{\lambda_1} f\|_{\Hb^1(\mathbb{R}^n)}^2 \frac{d\lambda_1}{\lambda_1^{n+1}} \cdots \frac{d\lambda_{m}}{\lambda_{m}^{n+1}}\right)^{1/2}.
\end{align*}
Now use the definition of the $\Hb^1(\mathbb{R}^n)$ norm to write
$$\|\Wc_{{\lambda_{m}}}U_{\lambda_{m-1}} \cdots U_{\lambda_1} f\|_{\Hb^1(\mathbb{R}^n)} = \| \Wc_{{\lambda_{m}}}U_{\lambda_{m-1}} \cdots U_{\lambda_1} f\|_{\Lb^1(\mathbb{R}^n)} + \sum_{j=1}^n\left\| \left(R_j\Wc_{{\lambda_{m}}}\right)(U_{\lambda_{m-1}} \cdots U_{\lambda_1} f)  \right\|_{\Lb^1(\mathbb{R}^n)}.$$
Thus, since $R_j\Wc_{{\lambda_{m}}}h = h \ast \left(R_j\psi_{\lambda_{m}}\right)$ and $R_j\psi$ wavelet, we can use our induction hypothesis and the previous lemma to get
\begin{align*}
    &C\left(\int_0^{\infty}\cdots \int_0^{\infty} \| \Wc_{{\lambda_{m}}} (U_{\lambda_{m-1}} \cdots U_{\lambda_1} f)\|_{\Hb^1(\mathbb{R}^n)}^2 \frac{d\lambda_1}{\lambda_1^{n+1}} \cdots \frac{d\lambda_{m}}{\lambda_{m}^{n+1}}\right)^{1/2}\\
    & \leq C\left(\int_0^{\infty}\cdots \int_0^{\infty} \| \Wc_{{\lambda_{m}}} (U_{\lambda_{m-1}} \cdots U_{\lambda_1} f)\|_{\Lb^1(\mathbb{R}^n)}^2 \frac{d\lambda_1}{\lambda_1^{n+1}} \cdots \frac{d\lambda_{m}}{\lambda_{m}^{n+1}}\right)^{1/2} \\
    &+ C\sum_{j=1}^n\left(\int_0^{\infty}\cdots \int_0^{\infty} \left\|\left(R_j\Wc_{{\lambda_{m}}}\right) (U_{\lambda_{m-1}} \cdots U_{\lambda_1} f) \right\|_{\Lb^1(\mathbb{R}^n)}^2 \frac{d\lambda_1}{\lambda_1^{n+1}} \cdots \frac{d\lambda_{m}}{\lambda_{m}^{n+1}}\right)^{1/2} \\
    &\leq C_{m+1}\|f\|_{\mathbb{H}^1(\mathbb{R}^n)}.
\end{align*} Thus, the theorem is proved by induction.
\end{proof}

The case of $n=1$ is a little trickier. We have the following multiplier property for the Hilbert Transform:
\begin{equation} \label{eqn: fourier transform of hilbert transform}
    \widehat{Hf}(\omega) = \left\{
    \begin{array}{ll}
         +i \hf (\omega) & \omega < 0  \\
         -i \hf (\omega) & \omega > 0
    \end{array}
    \right.
\end{equation} Unfortunately, this yields less regularity for $\widehat{Hf}$ at the origin without additional assumptions. However, notice that the Hilbert transform commutes with dilations, so in particular:
\begin{equation*}
    H (\psi_{\lambda}) = H(\psi)_{\lambda} \quad \text{and} \quad H(\psi_j) = H(\psi)_j \, .    
\end{equation*}
Using the calculation of $\widehat{Hf}$ in \eqref{eqn: fourier transform of hilbert transform} we see that
\begin{equation*}
    H \psi = -i \psi \, , \quad \text{if } \psi \text{ is complex analytic.}
\end{equation*} Thus, we have the following corollary.

\begin{corollary} \label{thm: m-layer L1 bounded norm R^1}
Let $\psi$ be a complex analytic wavelet such that \eqref{eqn: decay condition} and \eqref{eqn: holder condition} hold. Then for $f \in \Hb^1 (\R)$, there exists a constant $C_m$ such that
\begin{equation*}
    \| S_{\text{cont},1}^m f \|_{\Lb^2 (\R_+^m)} \leq C_m \|f\|_{\Hb^1 (\R)} \, .
\end{equation*}
Additionally, 
$$\| S_{\text{dyad},1}^m f \|_{\ellb^2 (\Z^m)} \leq C_m \| f \|_{\Hb^1 (\R)}.$$
\end{corollary}

\subsection{$\Lb^q (\R^n)$ Wavelet Scattering Transform}
In this subsection, assume $1 < q < 2$. We prove that for $m \in \mathbb{N}$, $S_{\text{cont},q}^m : \Lb^q (\R^n) \rightarrow \Lb^2 (\R_+^m)$. The norm for $S_{\text{cont},q}^m f$ is:
\begin{align*}
    \| S_{\text{cont},q}^m f \|_{\Lb^2 (\R_+^m)}^q &:= \left(\int_0^{\infty}\int_0^{\infty}\cdots \int_0^{\infty} | S_{\text{cont},q}^m f(\lambda_1, \lambda_2, \ldots, \lambda_m)|^2 \, \frac{d\lambda_1}{\lambda_1^{n+1}} \, \frac{d\lambda_2}{\lambda_2^{n+1}} \, \cdots \, \frac{d\lambda_m}{\lambda_m^{n+1}}\right)^{q/2} \\
    &= \left(\int_0^{\infty}\int_0^{\infty}\cdots \int_0^{\infty} \left(\left\| (U_{\lambda_{m-1}} \cdots U_{\lambda_1} f)\ast  \psi_{\lambda_{m}} \right\|_q\right)^2 \, \frac{d\lambda_1}{\lambda_1^{n+1}} \, \frac{d\lambda_2}{\lambda_2^{n+1}} \,  \cdots \, \frac{d\lambda_m}{\lambda_m^{n+1}}\right)^{q/2} \, .
\end{align*}There is also an analagous result for
\begin{equation*}
    \| S_{\text{dyad},q}^m f \|_{\ellb^2 (\Z^m)}^q := \left(\sum_{j_m \in \mathbb{Z}} \cdots \sum_{j_m \in \mathbb{Z}} | S_{\text{dyad},q}^m f(\lambda_1, \lambda_2, \ldots, \lambda_m)|^2 \right)^{q/2}.
\end{equation*}
\begin{theorem} \label{thm: m-layer Lq bounded norm}
Let $1 < q < 2$. Also, let $\psi$ be a wavelet that satisfies properties \eqref{eqn: decay condition} and \eqref{eqn: holder condition} and let $S_{\text{cont},q}^m$ and $S_{\text{dyad},q}^m$ be defined as above. Then there exists a universal constant $C_m > 0$ such that $\| S_{\text{cont},q}^mf \|_{\Lb^2(\R_+)}^q \leq C_m \|f \|_{q}^q$ for all $f \in \Lb^q (\R^n)$, and furthermore $\| S_{\text{dyad},q}^m f \|_{\ellb^2 (\Z)}^q \leq C_m \| f \|_{q}^q$.
\end{theorem}

\begin{proof} We proceed by induction and consider the case of $m = 1$ first. Let $f \in \Lb^q (\R^n)$. For the continuous wavelet transform, we apply Minkowski's integral inequality:
\begin{align*}
	\| S_{\text{cont},q}f \|_{\Lb^2(\R_+)}^q &= \left[\int_0^\infty \left(\| f \ast \psi_\lambda \|_q\right)^{q} \, \frac{d\lambda}{\lambda^{n+1}}\right]^{1/2} \\
	&= \left[\int_0^\infty \left(\int_{\R^n} |f \ast \psi_\lambda(x)|^q \ dx \right)^{2/q} \, \frac{d\lambda}{\lambda^{n+1}}\right]^{q/2} \\
    &\leq \int_{\R^n}\left(\int_{0}^\infty |f \ast \psi_\lambda(x)|^{2} \ \frac{d\lambda}{\lambda^{n+1}}\right)^{q/2} \ dx \\
    &= \|G(f)\|_q^q \\
    &\leq C\|f\|_q^q.
\end{align*}
where in the last inequality we used Theorem \ref{thm: LP g-function Lp bounded}.

Now, let us assume that 
\begin{equation*}
    \| S_{\text{cont},q}^{m} f \|_{\Lb^2 (\R_+^{m})}^q \leq C^{m \cdot q} \|f\|_{\Lb^q(\R^n)}^q \, .
\end{equation*}
We apply Minkowski's Integral inequality \cite{hardy1988inequalities} to swap and then bound:
\begin{align*}
\| S_{\text{cont},q}^{m+1} f \|_{\Lb^2 (\R_+^{m+1})}^q &=\left[\int_0^\infty\dots \int_0^\infty \left(\left\| (U_{\lambda_{1}} \cdots U_{\lambda_1} f)\ast  \psi_{\lambda_{m+1}} \right\|_q\right)^{2/q} \, \frac{d\lambda_{1}}{\lambda_1^{n+1}}\dots \frac{d\lambda_{m+1}}{\lambda_{m+1}^{n+1}}\right]^{q/2} \\
&=\left[\int_0^\infty\dots \int_0^\infty \left( \int_{\mathbb{R}^n}|(U_{\lambda_{1}} \cdots U_{\lambda_1} f)\ast  \psi_{\lambda_{m+1}}(x)|^q \, dx \right)^{2/q} \, \frac{d\lambda_{1}}{\lambda_1^{n+1}}\dots \frac{d\lambda_{m+1}}{\lambda_{m+1}^{n+1}}\right]^{q/2} \\
&=\left[\int_0^\infty\dots \int_0^\infty \left[\int_0^\infty \left( \int_{\mathbb{R}^n}|(U_{\lambda_{1}} \cdots U_{\lambda_1} f)\ast  \psi_{\lambda_{m+1}}(x)|^q \, dx \right)^{2/q} \, \frac{d\lambda_{m+1}}{\lambda_{m+1}^{n+1}}\right]^{\frac{q}{2} \cdot \frac{2}{q}} \frac{d\lambda_{1}}{\lambda_1^{n+1}}\dots \frac{d\lambda_{m}}{\lambda_{m}^{n+1}}\right]^{q/2} \\
&\leq 
\left[\int_0^\infty\dots \int_0^\infty \left[ \int_{\mathbb{R}^n}\left( \int_0^\infty |(U_{\lambda_{1}} \cdots U_{\lambda_1} f)\ast  \psi_{\lambda_{m+1}}(x)|^2 \, \frac{d\lambda_{m+1}}{\lambda_{m+1}^{n+1}} \right)^{q/2} \, dx \right]^{\frac{2}{q}} \frac{d\lambda_{1}}{\lambda_1^{n+1}}\dots \frac{d\lambda_{m}}{\lambda_{m}^{n+1}}\right]^{q/2} \\
&= \left[\int_0^\infty\dots \int_0^\infty \|G(U_{\lambda_{1}} \cdots U_{\lambda_1} f)\|_q^2 \frac{d\lambda_{1}}{\lambda_1^{n+1}}\dots \frac{d\lambda_{m}}{\lambda_{m}^{n+1}}\right]^{q/2} \\
& \leq C^{q} \left[\int_0^\infty\dots \int_0^\infty \|(U_{\lambda_{1}} \cdots U_{\lambda_1}) f\|_q^2 \frac{d\lambda_{1}}{\lambda_1^{n+1}}\dots \frac{d\lambda_{m}}{\lambda_{m}^{n+1}}\right]^{q/2} \\
&= C^{q} \| S_{\text{cont},q}^{m}f\|_{\Lb^2(\R_+^{m})}^q\\
& \leq C^{(m+1)q}  \|f\|_q^q.
\end{align*}
This proves the desired claim.
\end{proof}

\section{Stability to Dilations}
We now consider dilations defined by $\tau(x) = c x$ for some constant $c$, so that $L_\tau f(x) = f((1-c)x)$. We will start by proving a lemma that will be useful for our work.

\begin{lemma} \label{dilate wavelet} Assume $L_{\tau}$ is defined as above. Then
$$L_{\tau}f \ast \psi_{\lambda}(x) = (1-c)^{-n/2}\left(f \ast \psi_{(1-c)\lambda}\right)((1-c)x).$$
\end{lemma}
\begin{proof}
Notice that
$$L_{\tau}f \ast \psi_{\lambda}(x) = \int_{\R^n}f((1-c)y) \psi_{\lambda} (x-y)\, dy.$$
We make the substitution $z = (1-c)y$. Then it follows that
\begin{align*}
    L_{\tau}f \ast \psi_{\lambda}(x) &= (1-c)^{-n}\int_{\R^n} f(z) \psi_{\lambda}(x-(1-c)^{-1}z) \, dz\\
    &= (1-c)^{-n}\int_{\R^n} f(z) \lambda^{-n/2}\psi\left(\lambda^{-1}(x-(1-c)^{-1}z)\right)\, dz\\
    &= (1-c)^{-n/2}\int_{\R^n}f(z) [(1-c)\lambda]^{-n/2} \psi\left(\left[(1-c)\lambda\right]^{-1}\left((1-c)x - z\right)\right)\, dz \\
    &= (1-c)^{-n/2}\int_{\R^n}f(z) \psi_{(1-c)\lambda}\left((1-c)x - z\right)\, dz \\
    &= (1-c)^{-n/2}f\ast \psi_{(1-c)\lambda}\left((1-c)x\right) \\
    &= (1-c)^{-n/2}L_\tau\left(f \ast \psi_{(1-c)\lambda}\right)(x).
\end{align*}
\end{proof}
\begin{remark}
We also have
$$L_\tau \Wc_\lambda f (x) = (f\ast\psi_\lambda)(x(1-c)).$$
\end{remark}

Before we begin the next Lemma, we explain the general idea behind our approach to explain the necessity of Lemma \label{prop: newwavelet}. Define
\begin{equation}
\Psi(x) = (1-c)^{-n/2}\psi_{(1-c)}(x) - \psi(x).    
\end{equation} We want to prove that $\Psi$ satisfies (\ref{eqn: decay condition}) and (\ref{eqn: holder condition}) with a linear dependence on $c$ for future stability lemmas.

\begin{lemma}\label{prop: newwavelet}
Suppose that $\psi$ is a wavelet that satisfies the following three conditions:
\begin{align} \
	|\psi(x)| &\leq \frac{A}{(1+|x|)^{n+1+\alpha}} \,  \quad x\in\R^n ,\\
	|\nabla\psi(x)| &\leq \frac{A}{(1+|x|)^{n+1+\beta}} \,  \quad x\in\R^n ,\\
	\|D^2\psi(x)\|_\infty &\leq \frac{A}{(1+|x|)^{n+1+\kappa}} \,  \quad x\in\R^n,
\end{align} for $\alpha, \beta, \kappa > 0$. Consider
\begin{equation*}
\Psi(x) = (1-c)^{-n/2}\psi_{(1-c)}(x) - \psi(x).    
\end{equation*}
for $c < \frac{1}{2n}$. Then $\Psi$ is a wavelet satisfying (\ref{eqn: decay condition}) and (\ref{eqn: holder condition}).
\end{lemma}

\begin{proof}
Without loss of generality, assume $\alpha < \beta < \kappa <1$. First, it's clear that $\int_{\R^n}\Psi = 0$. We now just need to verify properties (\ref{eqn: decay condition}) and (\ref{eqn: holder condition}). Assume $c > 0$. We can modify the proof accordingly if $c < 0$. Then
\begin{align*}
|\Psi(x)| &= \left|(1-c)^{-n/2} \psi_{(1-c)}(x) - \psi(x) \right| \\
&= (1-c)^{-n}\left| \psi\left(\frac{x}{(1-c)}\right) - (1-c)^n\psi\left(x\right) \right| \\
&\leq (1-c)^{-n}\left|\psi\left(\frac{x}{1-c}\right) - \psi\left(\frac{1-c}{1-c}x\right)\right| +  (1-c)^{-n}\sum_{j=1}^n \binom{n}{j}c^j\left|\psi\left(x\right)\right|.
\end{align*} Now use mean value theorem on the first term to choose a point $z$ on the segment connecting $\frac{x}{1-c}$ and $x$ such that
$$\frac{c}{1-c}\left|[\nabla \psi(z)]^Tx\right| = \left|\psi\left(\frac{x}{1-c}\right) - \psi\left(\frac{1-c}{1-c}x\right)\right|.$$
We now use Cauchy-Schwarz to bound the left side:
$$\frac{c}{1-c}\left|[\nabla \psi(z)]^Tx\right| \leq \frac{c}{1-c} \frac{A|x|}{\left(1 + |z|\right)^{n +1+ \beta}}.$$
Since $z$ lies on the segment connecting $\frac{x}{1-c}$ and $x$, we see that for some $t \in [0,1]$, we have
\begin{align*}
z &= (1-t)\frac{x}{1-c} + tx \\
&= \frac{1-t}{1-c} x + \frac{t-tc}{1-c}x\\
&= \frac{1-t + t - tc}{1-c} x\\
&= \frac{1-tc}{1-c}x.
\end{align*} Thus, $|z| \geq |x|$.
It now follows that 
$$\frac{c}{1-c} \frac{A|x|}{\left(1 + |z|\right)^{n+ 1 + \beta}} \leq \frac{c}{1-c}\frac{A}{\left(1+|x|\right)^{n + \beta}}.$$
Finally, we get
\begin{align*}
|\Psi_\lambda(x)| &\leq \frac{c}{(1-c)^{n+1}} \frac{A}{\left(1 + |x|\right)^{n+\beta}} + \frac{\sum_{j=1}^n \binom{n}{j}c^j}{(1-c)^{n+1}} \frac{A}{\left(1 + |x|\right)^{n+\alpha}} \\
&\leq 2A\left(\frac{2n}{2n-1}\right)^{-n-1} \frac{\sum_{j=1}^n \binom{n}{j}c^j}{\left(1 + |x|\right)^{n+\alpha}} \\
& \leq \frac{A_n  c}{\left(1 + |x|\right)^{n+\alpha}}
\end{align*} for some constant $A_n$ since we assume $\alpha < \beta$ and $c < \frac{1}{2n}$. Thus, (\ref{eqn: decay condition}) is satisfied. \newline

We use a similar idea for proving (\ref{eqn: holder condition}) holds.  Assume $c > 0$ without loss of generality and further assume that $|x| \geq 2|y|$. By Mean Value Theorem, there exists $z$ on the line segment connecting $x$ and $x-y$ such that
 $$|\Psi(x-y) - \Psi(x)| = |\nabla\Psi(z)||y|.$$
Like before, we notice that 
\begin{align*}
|\nabla\Psi(z)| &= \left|(1-c)^{-n/2} \nabla\psi_{(1-c)}(z) - \nabla \psi(z) \right| \\
&= \left|(1-c)^{-n-1} \nabla\psi\left(\frac{z}{1-c}\right) - \nabla \psi(z) \right| \\
&= (1-c)^{-n-1}\left|\nabla\psi\left(\frac{z}{1-c}\right) - (1-c)^{n+1}\nabla \psi(z) \right| \\
&\leq (1-c)^{-n-1}\left|\nabla\psi\left(\frac{z}{1-c}\right) - \nabla\psi\left(\frac{1-c}{1-c}z\right)\right| +  (1-c)^{-n-1}\sum_{j=1}^{n+1} \binom{n+1}{j}c^j\left|\nabla\psi\left(z\right)\right|.
\end{align*} Let $S$ be the set of points on the segment connecting $\frac{z}{1-c}$ and $z$. By Mean Value Inequality, since $S$ is closed and bounded, we have
$$\left|\nabla\psi\left(\frac{z}{1-c}\right) - \nabla\psi\left(\frac{1-c}{1-c}z\right)\right| \leq \frac{c}{1-c}\max_{w \in S} \left\|D^2\psi(w)\right\|_\infty|z|.$$
The maximum for the quantity above is attained in $S$, so let us say the maximizer is $w_1 = (1-t)\frac{z}{1-c} + tz$ for some $t \in [0,1]$. Now use decay of the Hessian to bound the right side: 
$$\frac{c}{1-c}\max_{w \in S} \left\|D^2 \psi(w)\right\|_\infty|z| \leq \frac{c}{1-c} \frac{A |z|}{\left(1 + |w_1|\right)^{n +1+\kappa}}.$$
It follows that
\begin{align*}
w_1 &= (1-t)\frac{z}{1-c} + tz \\
&= \frac{1-t}{1-c} z + \frac{t-tc}{1-c}z\\
&= \frac{1-t + t - tc}{1-c} z\\
&= \frac{1-tc}{1-c}z.
\end{align*} Thus, $|w_1| \geq |z|.$
We conclude  
$$\frac{c}{1-c} \frac{A|z|}{\left(1 + |w_1|\right)^{n+ 1 + \kappa}} \leq \frac{c}{1-c}\frac{A}{\left(1+|z|\right)^{n  + \kappa}}.$$
For bounding $|\nabla \Psi(z)|$, we see 
\begin{align*}
|\nabla\Psi(z)| &\leq \frac{c}{(1-c)^{n+2}} \frac{A}{\left(1 + |z|\right)^{n+\kappa}} +  \frac{\sum_{j=1}^{n+1}\binom{n+1}{j} c^j}{(1-c)^{n+1}} \frac{A}{\left(1 + |z|\right)^{n+1+\beta}} \\
&\leq A(1-c)^{-n-2} \frac{2\sum_{j=1}^{n+1}\binom{n+1}{j} c^j}{\left(1 + |z|\right)^{n+\kappa}} \\
& \leq \left(\frac{2n}{2n-1}\right)^{n+2} \frac{2A \sum_{j=1}^{n+1}\binom{n+1}{j} c^j}{\left(1 + |z|\right)^{n+\kappa}}.
\end{align*} 

Going back to proving (\ref{eqn: holder condition}) holds for $\Psi$,  
$$|\Psi(x-y) - \Psi(x)| = |\nabla\Psi(z)||y| \leq \left(\frac{2n}{2n-1}\right)^{n+2} \frac{2A \sum_{j=1}^{n+1}\binom{n+1}{j} c^j |y|}{\left(1 + |z|\right)^{n+\kappa}}.$$
since the point $z$ lies on the lines on a line segment connecting $x-y$ and $x$ with $|x| \geq 2|y|$, we can use an argument similar to above to conclude 
\begin{align*}
|\Psi(x-y) - \Psi(x)| &\leq 2^{n+1+\kappa}\left(\frac{2n}{2n-1}\right)^{n+2} \frac{A \sum_{j=1}^{n+1}\binom{n+1}{j}c^j}{\left(1 + |x|\right)^{n+\kappa}} |y|.
\end{align*}
Now integrate to get 
\begin{align*}
\int_{|x| \geq 2|y|} |\Psi(x-y) - \Psi(x)| \, dx &\leq 2^{n+1+\kappa}\left(\frac{2n}{2n-1}\right)^{n+2}A\sum_{j=1}^{n+1} \binom{n+1}{j}c^j |y|\int_{|x| \geq 2|y|} \frac{dx}{|x|^{n+\kappa}}\\
&= 2^{n+1+\kappa}\left(\frac{2n}{2n-1}\right)^{n+2}A I_n \sum_{j=1}^{n+1} \binom{n+1}{j} c^j |y|^{1-\kappa},
\end{align*} where $I_n$ is some constant associated with the integration. Finally, we have a bound of
$$\int_{|x| \geq 2|y|} |\Psi(x-y) - \Psi(x)| \, dx \leq \tilde{A}_n c |y|^{1-\kappa}.$$
for some constant $\tilde{A}_n$ only dependent on the dimension $n$. Thus, (\ref{eqn: holder condition}) holds with exponent $1 - \kappa \in (0,1)$. Let $\hat{A}_n  = \max\{A_n, \tilde{A}_n\}.$ It follows that 
\begin{align*}
&|\Psi_\lambda(x)| \leq \frac{\hat{A}_n  c}{\left(1 + |x|\right)^{n+\alpha}} \\
&\int_{|x| \geq 2|y|} |\Psi(x-y) - \Psi(x)| \, dx \leq \hat{A}_n c |y|^{1-\kappa}.
\end{align*}
\end{proof}

It follows from Problem 6.1.2 in \cite{grafakos} that the bound in the $G$-function depends linearly on the constant $A$ from Theorem \ref{thm: LP g-function Lp bounded} when proving $\Lb^2(\mathbb{R}^n)$ boundedness. Thus, the following corollaries hold.

 \begin{corollary}
 Assume $|c| < \tfrac{1}{2n}$. For $\psi$ satisfying the conditions of Lemma \ref{prop: newwavelet}, when $1 < p < \infty$, there exist constants $C_{n,p}$ and $\hat{C}_{n,p}$ such that 
 $$\left\|\left(\int_{0}^\infty |f \ast \Psi_\lambda(x)|^2 \, \frac{d \lambda}{\lambda^{n+1}}\right)^{1/2}\right\|_{\Lb^p(\mathbb{R}^n)} \leq c \cdot C_{n,p} \max\{p, (p-1)^{-1}\} \|f\|_{\Lb^p(\mathbb{R}^n)}$$
 and 
 $$\left\|\left(\sum_{j \in \mathbb{Z}} |f \ast \Psi_j(x)|^2\right)^{1/2}\right\|_{\Lb^p(\mathbb{R}^n)} \leq  c \cdot \hat{C}_{n} \max\{p, (p-1)^{-1}\} \|f\|_{\Lb^p(\mathbb{R}^n)}.$$
Alternatively, if one of the following holds:
\begin{itemize}
    \item $n = 1$, $\psi$ is complex analytic and satisfies the conditions of Lemma \ref{prop: newwavelet},
    \item $n \geq 2$ and $\psi$ satisfies the conditions of Lemma \ref{lem: riesz transform decay},
\end{itemize} there exist constants $H_n$ and $\hat{H}_n$ such that
 $$\left\|\left(\int_{0}^\infty |f \ast \Psi_\lambda(x)|^2 \, \frac{d \lambda}{\lambda^{n+1}}\right)^{1/2}\right\|_{\Lb^1(\mathbb{R}^n)} \leq c \cdot H_n \|f\|_{\Hb^1(\mathbb{R}^n)}$$
and 
$$\left\|\left(\sum_{j \in \mathbb{Z}} |f \ast \Psi_j(x)|^2 \right)^{1/2}\right\|_{\Lb^1(\mathbb{R}^n)} \leq c \cdot \hat{H}_n \|f\|_{\Hb^1(\mathbb{R}^n)}.$$
 \end{corollary}

 Now we can use the results above for our main dilation stability results.
\begin{theorem} \label{dilate m layers wavelet}
Suppose that $\psi$ is a wavelet that satisfies the conditions of Lemma \ref{prop: newwavelet}. Then there exists a constants $K_{n,m}$ and $\hat{K}_{n,m}$ only dependent on $n$ and $m$ such that
\begin{equation*}
    \|S_{\text{cont},2}^mf - S_{\text{cont},2}^mL_\tau f \|_{\Lb^2(\R_+^m)} \leq |c| \cdot K_{n,m} \|f\|_2
\end{equation*} and
\begin{equation*}
    \|S_{\text{dyad},2}^mf - S_{\text{dyad},2}^mL_\tau f \|_{\Lb^2(\R_+^m)} \leq |c| \cdot \hat{K}_{n,m} \|f\|_2
\end{equation*}
 for any $|c| < \tfrac{1}{2n}$.
\end{theorem}

\begin{proof} Without loss of generality, assume $c > 0$. Let
\begin{align*}
\Wc_tf &= f \ast \psi_t \\
Mf &= |f|\\
U_t &= M\Wc_t \\
A_qf &= \left(\int_{\mathbb{R}^n} f^q(x) \, dx\right)^{1/q}.
\end{align*} It follows that $S_{\text{cont},2}^m = A_2MW_{\lambda_m}U_{\lambda_{m-1}} \cdots U_{\lambda_1}$. We will also let $V_{m-1} = U_{\lambda_{m-1}} \cdots U_{\lambda_1}$, with $V_0$ being the identity operator, and make a slight abuse of notation by denoting $\Wc_{\lambda_m}$ as $\Wc$. First, we will add and subtract $A_2ML_\tau \Wc V_{m-1}f$ and apply triangle inequality:
\begin{align*}
\|S_{\text{cont},2}^{m}f - S_{\text{cont},2}^{m}L_\tau f\|_{\Lb^2 (\R_+^m)} &= \|A_2M\Wc V_{m-1} f -  A_2M\Wc V_{m-1}L_\tau f\|_{\Lb^2 (\R_+^m)} \\
&\leq\|A_2M\Wc V_{m-1}f - A_2ML_\tau \Wc V_{m-1}f \|_{\Lb^2 (\R_+^m)} \\
&+ \|A_2ML_\tau \Wc V_{m-1}f-A_2M\Wc V_{m-1}, L_\tau f\|_{\Lb^2 (\R_+^m)}.
\end{align*}
We'll start by bounding the first term. We see that $g = \Wc V_{m-1}f \in \mathbb{\Lb}^2(\mathbb{R}^n)$. Thus
$$|A_2M\Wc V_{m-1}f - A_2ML_\tau \Wc V_{m-1}f| = \left| \|g\|_2 - \|L_\tau g\|_2\right|.$$
Now use a change of variables:
$$\|L_\tau g\|_2^2 = \int_{\mathbb{R}^n} |g((1-c)x)|^2 \, dx = (1-c)^{-n} \|g\|_2^2.$$
It then follows that 
    $$\left|\|L_\tau g\|_2 - \|g\|_2\right| = \|g\|_2\left(\frac{1}{(1-c)^{n/2}}-1\right) \leq   \|g\|_2\left(\frac{1}{(1-c)^{n}}-1\right).$$
Taking the scattering norm yields
\begin{align*}
\|A_2M\Wc V_{m-1}f - A_2ML_\tau \Wc V_{m-1}f\|_{\Lb^2 (\R_+^m)}^2
& \leq \left(\frac{1}{(1-c)^n}-1\right)^2\|S_{\text{cont},2}^mf\|_{\Lb^2(\R_+^m)}^2 \\
& = \left(\frac{1-(1-c)^n}{(1-c)^n}\right)^2\|S_{\text{cont},2}^mf\|_{\Lb^2(\R_+^m)}^2 \\
& = \left(\frac{1}{(1-c)^n} \sum_{j=1}^n \binom{n}{j} c^j\right)^2 \|S_{\text{cont},2}^mf\|_{\Lb^2(\R_+^m)}^2 \\
& \leq \left[\left(\frac{2n}{2n-1}\right)^{n}\sum_{j=1}^n \binom{n}{j} c^j\right]^2 \|S_{\text{cont},2}^mf\|_{\Lb^2(\R_+^m)}^2\\
& \leq c^2 \cdot C_{m,n}\|f\|_2^2.
\end{align*}
For the second term, apply Minkwoski's inequality for $2$ norms:
\begin{align*}
&\|A_2ML_\tau \Wc V_{m-1}f-A_2M\Wc V_{m-1} L_\tau f\|_{\Lb^2 (\R_+^m)} \\
&= \left(\int_{0}^\infty \cdots \int_{0}^\infty \left|\|L_\tau \Wc V_{m-1}f\|_2 - \|\Wc L_\tau V_{m-1}f\|_2\right|^2 \frac{d \lambda_1}{\lambda_1^{n+1}} \cdots \frac{d \lambda_m}{\lambda_{m}^{n+1}}\right)^{1/2}\\
&\leq \left(\int_{0}^\infty \cdots \int_{0}^\infty \|L_\tau \Wc V_{m-1}f-\Wc L_\tau V_{m-1}f\|_2^2 \frac{d \lambda_1}{\lambda_1^{n+1}} \cdots \frac{d \lambda_m}{\lambda_{m}^{n+1}}\right)^{1/2}\\
&= \|A_2 M [\Wc V_{m-1}, L_\tau] f \|_{\Lb^2 (\R_+^m)}.
\end{align*}
Now this is a commutator term, and we can now bound:
\begin{align*}
\|A_2M[\Wc V_{m-1}, L_\tau]f\|_{\Lb^2 (\R_+^m)}^2 &= \int_0^\infty \cdots \int_0^\infty \|[\Wc V_{m-1}, L_\tau] f \|_2^{2}\,\frac{d\lambda_1}{\lambda_1^{n+1}} \cdots \frac{d\lambda_m}{\lambda_m^{n+1}} \\
&=\||[\Wc V_{m-1}, L_\tau]f \|^2_{\Lb^2(\R_{+}^m \times \R^n)} \\
&\leq  \|[\Wc V_{m-1}, L_\tau]\|^2_{\Lb^2(\R_{+}^m \times \R^n) \to  \Lb^2(\R^n)} \|f\|_2^2.
\end{align*}
We examine the commutator term more closely. Without a loss of generality, assume $m \geq 2$. By expanding it, we see that each term contains $[\Wc,L_\tau]$. It follows that 
\begin{align*}
\norm{[\Wc V_{m-1},L_\tau]}_{\Lb^2(\R_{+}^m \times \R^n)} &\leq m \norm{\Wc}_{\Lb^2(\R_{+} \times \R^n) \to  \Lb^2(\R^n)} ^{m-1}\norm{M}_{\Lb^2(\R^n) \to \Lb^2(\R^n)}^{m-1}\norm{[\Wc,L_\tau]}_{\Lb^2(\R_+ \times\R^n) \to  \Lb^2(\R^n)}  \\
&\leq C_m \norm{[\Wc,L_\tau]}_{\Lb^2(\R_+ \times\R^n) \to  \Lb^2(\R^n)} .
\end{align*}

Thus, once we bound this quantity appropriately, we will finish the proof. We start by writing
\begin{equation*}
    \|[\Wc, L_\tau]f \|_{\Lb^2(\R_{+} \times \R^n)}^2 = \int_{0}^\infty \|(L_{\tau}f) \ast \psi_{\lambda} - L_\tau\left(f \ast \psi_{\lambda} \right)\|_2^2 \, \frac{d\lambda}{\lambda^{n+1}} \, .
\end{equation*}
By substitution with $z = (1-c)x$ and Lemma \ref{dilate wavelet},
\begin{align*}
    \|(L_{\tau}f) \ast \psi_{\lambda} - L_\tau\left(f \ast \psi_{\lambda} \right)\|_2^{2} &= \int_{\R^n}\left|\right(L_{\tau}f \ast \psi_{\lambda} \left)(x) - L_\tau\left(f \ast \psi_{\lambda} \right)(x)\right|^2 \, dx\\
    &= \int_{\R^n} \left|(1-c)^{-n/2}\left(f \ast \psi_{(1-c)\lambda}\right)((1-c)x) - \left(f \ast \psi_{\lambda} \right)((1-c)x) \right|^2 \, dx \\
    &= (1-c)^{-n} \int_{\R^n} \left|(1-c)^{-n/2}\left(f \ast \psi_{(1-c)\lambda}\right)(z) - \left(f \ast \psi_{\lambda} \right)(z) \right|^2 \, dz\\
    &= (1-c)^{-n}\int_{\R^n} \left|f \ast \left((1-c)^{-n/2}\psi_{(1-c)\lambda} - \psi_{\lambda} \right)\right|^2 \, dz\\
    &= (1-c)^{-n}\int_{\R^n} \left|\left(f \ast \Psi_{\lambda} \right)(z)\right|^2 \, dz, \\
    &= (1-c)^{-n} \| f \ast \Psi_{\lambda} \|_2^2 \, .
\end{align*} 
Thus, we obtain
\begin{align*}
\int_{0}^\infty\|(L_{\tau}f) \ast \psi_{\lambda} - L_\tau\left(f \ast \psi_{\lambda} \right)\|_2^{2} \, \frac{d\lambda}{\lambda^{n+1}} &= (1-c)^{-n}\int_{0}^\infty\|f \ast \Psi_{\lambda} \|_2^{2}\frac{d\lambda}{\lambda^{n+1}} \\
&= (1-c)^{-n}\int_{\mathbb{R}^n}\int_{0}^\infty |f \ast \Psi_{\lambda}(x)|^2\,\frac{d\lambda}{\lambda^{n+1}}\, dx\\
&= (1-c)^{-n} \left\|\left(\int_{0}^\infty |f \ast \Psi_{\lambda}(x)|^2\,\frac{d\lambda}{\lambda^{n+1}}\right)^{1/2} \right\|_2^2\\
&\leq c^2 \cdot \left(\frac{2n}{2n-1}\right)^{n} C_{n,p}\|f\|_2^{2}.
\end{align*} 

It follows that 
$$\|S_{\text{cont},2}^{m}f - S_{\text{cont},2}^{m}L_\tau f\|_{\Lb^2 (\R_+^m)} \leq |c| \cdot K_{n,m} \|f\|_2$$
for any $c < \tfrac{1}{2n}$.
\end{proof}

As is customary at this point, we have the following corollaries. We start with the case where $1 < q < 2$.
\begin{corollary} \label{corollary: m layer dilation stability for q != 2} Assume $|c| < \tfrac{1}{2n}$. For $q \in (1, 2)$, there exists constants $K_{n,m,q}$ and $\hat{K}_{n,m,q}$ such that 
$$\|S_{\text{cont},q}^m f - S_{\text{cont},q}^m L_\tau f \|_{\Lb^2(\R_{+}^m)}^q \leq  |c|^q \cdot K_{n,m,q} \|f\|_q^q$$
and
$$\|S_{\text{dyad},q}^m f - S_{\text{dyad},q}^m L_\tau f \|_{\ellb^2(\Z^m)}^q \leq  |c|^q \cdot \hat{K}_{n,m,q} \|f\|_q^q.$$
\end{corollary}
\begin{proof}
Without loss of generality again, assume $c > 0$. First, we will add and subtract $A_qML_\tau \Wc V_{m-1}f$ and apply triangle inequality:
\begin{align*}
\|S_{\text{cont},q}^{m}f - S_{\text{cont},q}^{m}L_\tau f\|_{\Lb^2 (\R_+^m)} &= \|A_qM\Wc V_{m-1} f -  A_qM\Wc V_{m-1}L_\tau f\|_{\Lb^2 (\R_+^m)} \\
&\leq\|A_qM\Wc V_{m-1}f - A_qML_\tau \Wc V_{m-1}f \|_{\Lb^2 (\R_+^m)} \\
&+ \|A_qML_\tau \Wc V_{m-1}f-A_qM\Wc V_{m-1}, L_\tau f\|_{\Lb^2 (\R_+^m)}.
\end{align*}

We'll start by bounding the first term again. Define $g = \Wc V_{m-1}f \in \mathbb{\Lb}^q(\mathbb{R}^n)$. and we have
$$|A_qM\Wc V_{m-1}f - A_qML_\tau \Wc V_{m-1}f| = \left| \|g\|_q - \|L_\tau g\|_q\right|.$$
By change of variables, 
$$\left| \|g\|_q - \|L_\tau g\|_q\right| = \|g\|_q\left(\frac{1}{(1-c)^{n/q}}-1\right) \leq \|g\|_q\left(\frac{1}{(1-c)^{n}}-1\right).$$ 
Again, we have
\begin{align*}
\|A_qM\Wc V_{m-1}f - A_qML_\tau \Wc V_{m-1}f\|_{\Lb^2 (\R_+^m)}^q &\leq \left(\frac{1}{(1-c)^{n/q}}-1\right)^q \|S_{\text{cont},2}^mf\|_{\Lb^2(\R_+^m)}^q \\
& \leq \left(\frac{1}{(1-c)^n}-1\right)^q\|S_{\text{cont},q}^mf\|_{\Lb^2(\R_+^m)}^q \\
& = \left[\frac{1-(1-c)^n}{(1-c)^n}\right]^q\|S_{\text{cont},q}^mf\|_{\Lb^2(\R_+^m)}^q \\
& = \left[\frac{1}{(1-c)^n} \sum_{j=1}^n \binom{n}{j} c^j\right]^q \|S_{\text{cont},q}^mf\|_{\Lb^2(\R_+^m)}^q \\
& \leq \left[\left(\frac{2n}{2n-1}\right)^{n}\sum_{j=1}^n \binom{n}{j} c^j\right]^q \|S_{\text{cont},q}^mf\|_{\Lb^2(\R_+^m)}^q\\
& \leq |c|^q \cdot C_{m,n}\|f\|_q^q.
\end{align*}
For the second term, apply Minkoski's inequality for $q$ norms:
\begin{align*}
&\|A_qML_\tau \Wc V_{m-1}f-A_qM\Wc V_{m-1}, L_\tau f\|_{\Lb^2 (\R_+^m)} \\
&= \left(\int_{0}^\infty \cdots \int_{0}^\infty \left|\|L_\tau \Wc V_{m-1}f\|_q - \|\Wc L_\tau V_{m-1}f\|_q\right|^2 \frac{d \lambda_1}{\lambda_1^{n+1}} \cdots \frac{d \lambda_m}{\lambda_{m}^{n+1}}\right)^{1/2}\\
&\leq \left(\int_{0}^\infty \cdots \int_{0}^\infty \|L_\tau \Wc V_{m-1}f-\Wc L_\tau V_{m-1}f\|_q^2 \frac{d \lambda_1}{\lambda_1^{n+1}} \cdots \frac{d \lambda_m}{\lambda_{m}^{n+1}}\right)^{1/2}\\
&= \|A_q M [\Wc V_{m-1}, L_\tau] f \|_{\Lb^2 (\R_+^m)}.
\end{align*}
Via a similar reduction technique  for Theorem \ref{dilate m layers wavelet}, we can reduce to a commutator bound $\|A_q M [\Wc, L_\tau] f \|_{\Lb^2 (\R_+^m)}.$ Additionally, we have 
\begin{align*}
    \|(L_{\tau}f) \ast \psi_{\lambda} - L_\tau\left(f \ast \psi_{\lambda} \right)\|_q^{q} &= (1-c)^{-n} \| f \ast \Psi_{\lambda} \|_q^q \, .
\end{align*} 
Thus,
\begin{align*}
\|A_q M [\Wc, L_\tau] f \|_{\Lb^2 (\R_+^m)}^q &= \left(\int_{0}^\infty\|(L_{\tau}f) \ast \psi_{\lambda} - L_\tau\left(f \ast \psi_{\lambda} \right)\|_q^{2} \, \frac{d\lambda}{\lambda^{n+1}}\right)^{q/2} \\
&=  (1-c)^{-n}\left(\int_{0}^\infty\|f \ast \Psi_{\lambda} \|_q^{2}\frac{d\lambda}{\lambda^{n+1}}\right)^{q/2} \\
&\leq (1-c)^{-n} \left\|\left(\int_{0}^\infty |f \ast \Psi_{\lambda}(x)|^2\,\frac{d\lambda}{\lambda^{n+1}}\right)^{1/2} \right\|_q^{q}\\
& \leq |c|^q \cdot \tilde{C}_{n} \|f\|_q^{q}.
\end{align*} 
It follows that 
$$\|S_{\text{cont},q}^{m}f - S_{\text{cont},q}^{m}L_\tau f\|_{\Lb^2 (\R_+^m)}^q \leq |c|^q \cdot K_{n,m} \|f\|_q^q$$
for any $|c| < \tfrac{1}{2n}$.
\end{proof}

Additionally, for the case of $q  = 1$, we have the following corollaries that we will state, but not prove, since the idea is the same as the previous corollary.
\begin{corollary}
Suppose one of the following holds:
\begin{itemize}
    \item $n = 1$, $\psi$ is complex analytic and satisfies the conditions of Lemma \ref{prop: newwavelet},
    \item $n \geq 2$ and $\psi$ satisfies the conditions of Lemma \ref{lem: riesz transform decay},
\end{itemize} then there exist constants $K_{H,m}$ and $\hat{K}_{H,m}$ such that
$$\|S_{\text{cont},1}^m f - S_{\text{cont},1}^m L_\tau f \|_{\Lb^2(\R_{+}^m)} \leq c \cdot K_{H,m} \|f\|_{\Hb^1(\mathbb{R}^n)}$$
and 
$$\|S_{\text{dyad},1}^mf - S_{\text{dyad},1}^m L_\tau f \|_{\ellb^2(\Z^m)} \leq  c \cdot \hat{K}_{H,m} \|f\|_{\Hb^1(\mathbb{R}^n)}.$$
\end{corollary}

\section{Stability to Diffeomorphisms}
We now focus on the stability of $S_{\text{cont},q}^{m} f$ for general diffeomorphisms with $\|D\tau\|_\infty < \frac{1}{2n}$. The corresponding operator for diffeomorphisms is defined as $L_{\tau}f(x) = f(x - \tau(x))$. 

\subsection{Stability to Diffeomorphisms When $q = 2$}
% \begin{lemma} \label{lemma: commutator equivalence}
% Define $\phi_t(x) = t^{n/2}\psi(t x)$. We have
% $$\norm{[\Wc,L_\tau]f}_2^2 = \int_0^\infty \int_{\mathbb{R}^n} |(\phi_{t} \ast L_\tau f)(x) - L_\tau(\phi_{t} \ast f)(x)|^2 \, dx  \, t^{n-1} \, dt.$$
% \end{lemma}
% \begin{proof} The proof follows by a straightforward calculation:
% 	\begin{align*}
% 	\norm{[\Wc,L_\tau]f}_2^2 &= \int_0^\infty \norm{[\Wc_\lambda,L_\tau]f}_2^2\ \frac{d\lambda}{\lambda^{n+1}} \\
% 	&= \int_0^\infty \|\psi_\lambda \ast (L_\tau f) - L_\tau(\psi_\lambda \ast f)  \|_2^2 \frac{d\lambda}{\lambda^{n+1}}. 
% 	\end{align*}
% Now notice that $\psi_{\frac{1}{t}}(x) = t^{-n/2}\psi(tx) = \phi_t(x).$ Let $\lambda = \frac{1}{t}$. Then we have
% \begin{align*}
% \norm{[\Wc,L_\tau]f}_2^2 &= \int_0^\infty \|\psi_{\frac{1}{t}} \ast (L_\tau f) - L_\tau(\psi_{\frac{1}{t}} \ast f)\|_2^2 t^{n-1} \, dt \\
% &= \int_0^\infty  \|\phi_{t} \ast (L_\tau f) - L_\tau(\phi_{t} \ast f)\|^2_2 \, t^{n-1} \, dt.
% \end{align*}
% \end{proof}

% We can now switch normalizations for the next proposition. Define $\psi_\lambda(x) = \lambda^{n/2}\psi(\lambda x)$. Then define the operator $\Wc_\lambda f = \psi_\lambda \ast f$. We will bound 
% $$\norm{[\Wc,L_\tau]f}_2^2 :=  \int_0^\infty \norm{[\Wc_\lambda,L_\tau]f}_2^2\ \lambda^{n-1}\, d\lambda.$$ 
% using our new normalization, and it will be equivalent to bounding the commutator with our previous normalization because of the lemma above. 

\begin{proposition}
	\label{lem:Cont214}
    Assume $\psi$ and its first and second order derivatives have decay\footnote{Similar to \cite{nicola}, we have found that there needs to be $O((1+|x|)^{-n-2 + \alpha})$ decay for some $\alpha > 0$ to bound (E.26) in \cite{mallat2012group}.} in $O((1+|x|)^{-n-3})$, and $\int_{\mathbb{R}^n} \psi(x)\ dx=0$. Then for every $\tau \in C^2(\mathbb{R}^n)$ with $\norm{D\tau}_{\infty} \leq \frac{1}{2n}$, there exists $\tilde{C}_n>0$ such that:
	\begin{align*}
	\norm{[\Wc,L_\tau]}_{\Lb^2(\R_{+} \times \R^n) \to  \Lb^2(\R^n)}\leq \tilde{C}_n \left(\norm{D\tau}_{\infty}\left(\log\frac{\norm{\Delta\tau}_\infty}{\norm{D\tau}_\infty} \vee 1\right)+\norm{D^2\tau}_\infty\right).
	\end{align*}
\end{proposition}
\begin{proof}
The argument is a continuous version of Lemma 2.14 in \cite{mallat2012group}. We will first show how to transform our commutator term into an analogous commutator term from \cite{mallat2012group}. To shorten notation, we will denote $\norm{[\Wc,L_\tau]}_{\Lb^2(\R_{+} \times \R^n)}$ as $\norm{[\Wc,L_\tau]}$. We have
 \begin{align*}
	\norm{[\Wc,L_\tau]f}_{\Lb^2(\R_{+}\times \R^n)}^2 &= \int_0^\infty \norm{[\Wc_t,L_\tau]f}_2^2\ \frac{d t}{t^{n+1}} \\
	&= \int_0^\infty \|\psi_t \ast (L_\tau f) - L_\tau(\psi_t \ast f)  \|_2^2 \frac{dt}{t^{n+1}} \\
    &= \int_0^\infty \int_{\mathbb{R}^n} \left|\psi_t \ast (L_\tau f) - L_\tau(\psi_t \ast f)\right|^2 \, dx \, \frac{dt}{t^{n+1}}.
	\end{align*}
Notice that $\psi_\frac{1}{t}\left(x\right) = t^{n/2}\psi(t x)$. Use the change of variables $t = \frac{1}{\lambda}$ to get
\begin{align*}
\norm{[\Wc,L_\tau]f}_{\Lb^2(\R_{+} \times \R^n)}^2 &= \int_0^\infty \left\|\psi_{\frac{1}{\lambda}} \ast (L_\tau f) -  L_\tau( \psi_{\frac{1}{\lambda}} \ast f)\right\|_2^2 \, \lambda^{n-1}\, d\lambda \\
&= \int_0^\infty \left\|\lambda^{n/2}\psi_{\frac{1}{\lambda}} \ast (L_\tau f) -  L_\tau(\lambda^{n/2}\psi_{\frac{1}{\lambda}} \ast f)\right\|_2^2 \, \frac{d \lambda}{\lambda}.
\end{align*}
Define $\mathscr{W}_\lambda f = f \ast \lambda^{n/2}\psi_{\frac{1}{\lambda}}$ with $\lambda^{n/2}\psi_{\frac{1}{\lambda}}(x) = \lambda^{n}\psi(\lambda x).$ In other words, $\mathscr{W}_t$ is a convolution with an $\Lb^1$ normalized wavelet, which matches with the normalization in \cite{mallat2012group}. Now we have 
	\begin{align*}
	\norm{[\Wc,L_\tau]f}_{\Lb^2(\R_{+} \times \R^n)}^2 &= \int_0^\infty \norm{[\mathscr{W}_\lambda,L_\tau]f}_2^2 \, \frac{d\lambda}{\lambda}.
	\end{align*}
	This implies 
	\begin{align*}
	[\Wc,L_\tau]^*[\Wc,L_\tau] &= \int_0^\infty [\mathscr{W}_\lambda,L_\tau]^*[\mathscr{W}_\lambda,L_\tau] \, \frac{d\lambda}{\lambda}
	\end{align*}
	Defining $K_\lambda = \mathscr{W}_\lambda - L_\tau \mathscr{W}_\lambda L_\tau^{-1}$ so that $[\mathscr{W}_\lambda,L_\tau] =K_\lambda L_\tau$, we have:
	\begin{align*}
	\norm{[\Wc,L_\tau]} &= \norm{[\Wc,L_\tau]^*[\Wc,L_\tau]}^{1/2} \\
	&= \norm[\bigg]{\int_0^\infty [\mathscr{W}_\lambda,L_\tau]^*[\mathscr{W}_\lambda,L_\tau] \, \frac{d\lambda}{\lambda}}^{1/2} \\
	&= \norm[\bigg]{\int_0^\infty L_\tau^*K_\lambda^*K_\lambda L_\tau\ \frac{d\lambda}{\lambda}}^{1/2} \\
	&\leq \norm{L_\tau}\cdot\norm[\bigg]{\int_0^\infty K_\lambda^*K_\lambda \ \frac{d\lambda}{\lambda}}^{1/2},
	\end{align*}
    Since $\norm{L_\tau f}_2^2 \leq \left(\frac{1}{1-n\norm{D\tau}_\infty}\right) \norm{f}_2^2$, $$\norm{L_\tau} \leq \frac{1}{1-n\norm{D\tau}_\infty} \leq 2$$ and the problem is reduced to bounding $\left\|\int_0^\infty K_\lambda^*K_\lambda\ \lambda^{-1}\ d\lambda\right\|^{1/2}$. Let $\gamma \geq 1$. The integral is divided into three pieces:
	\begin{align*}
	\norm[\bigg]{\int_0^\infty K_\lambda^*K_\lambda \ \frac{d\lambda}{\lambda}}^{1/2} 
	&\leq \left(\norm[\bigg]{\int_0^{2^{-\gamma}} K_\lambda^*K_\lambda \ \frac{d\lambda}{\lambda}}+\norm[\bigg]{\int_{2^{-\gamma}}^1 K_\lambda^*K_\lambda \ \frac{d\lambda}{\lambda}}+\norm[\bigg]{\int_1^\infty K_\lambda^*K_\lambda \ \frac{d\lambda}{\lambda}}\right)^{1/2} \\
	&\leq \norm[\bigg]{\int_0^{2^{-\gamma}} K_\lambda^*K_\lambda \ \frac{d\lambda}{\lambda}}^{1/2}+\norm[\bigg]{\int_{2^{-\gamma}}^1 K_\lambda^*K_\lambda \ \frac{d\lambda}{\lambda}}^{1/2}+\norm[\bigg]{\int_1^\infty K_\lambda^*K_\lambda \ \frac{d\lambda}{\lambda}}^{1/2}\\
	&= P_1 + P_2 + P_3.
	\end{align*}
	
	To bound $P_1$, we decompose $K_\lambda = \tilde{K}_{\lambda,1} +\tilde{K}_{\lambda,2}$, where the kernels defining $\tilde{K}_{\lambda,1}, \tilde{K}_{\lambda,2}$ are
	\begin{align*}
	\tilde{k}_{\lambda,1}(x,u) &:= (1- \det(I - D\tau(u))) \lambda^{n} \psi(\lambda(x-u)) \\
    &:= a(u) \lambda^{n} \psi(\lambda(x-u)),\\
	\tilde{k}_{\lambda,2}(x,u) &:=\det(I - D\tau(u))(\lambda^{n}\psi(\lambda(x-u))-\lambda^{n}\psi(\lambda(x-\tau(x)-u+\tau(u))),
	\end{align*}
	respectively. Since our normalization matches with \cite{mallat2012group}, E.13 implies that there exists a constant $C_n$ such that
	$$\norm{\tilde{K}_{\lambda,2}} \leq C_n\lambda \norm{\Delta\tau}_\infty.$$
	We want to prove that
	$$\norm[\bigg]{\int_0^{1} \tilde{K}_{\lambda,1}^*\tilde{K}_{\lambda,1} \ \frac{d\lambda}{\lambda}}^{1/2}  \leq C_n\norm{D\tau}_\infty.$$
    Let $f \in \Lb^2(\mathbb{R}^n)$ be arbitrary and define $\tilde{\psi}(t) = \psi^{*}(-t)$. Based on \cite{mallat2012group}, the kernel of $K^{*}_{\lambda, 1}K_{\lambda, 1}$ is given by 
    $$\tilde{k}_\lambda(y,z) := a(y)a(z) \lambda^{n/2}\tilde{\psi}_{\tfrac{1}{\lambda}} \ast \lambda^{n/2}\tilde{\psi}_{\tfrac{1}{\lambda}}(z-y).$$
    Thus, it is sufficient to bound the quantity
    $$\int_{0}^1 \|K_{\lambda, 1}^{*}K_{\lambda, 1}f\|_2^2 \frac{d\lambda}{\lambda}.$$
    We see that $\|a\|_\infty \leq n \|D\tau\|_\infty.$ Substituting in the kernel and bounding yields
    \begin{align*}
        \int_{0}^1 \|K_{\lambda, 1}^{*}K_{\lambda, 1}f\|_2^2 \frac{d\lambda}{\lambda} &= \int_{0}^1\int_{\mathbb{R}^n} \left| \int_{\mathbb{R}^n}  a(y)a(z) \left(\lambda^{n/2}\tilde{\psi}_{\tfrac{1}{\lambda}} \ast \lambda^{n/2}{\psi}_{\tfrac{1}{\lambda}}\right) (z-y) f(y)  \, dy \right|^2  \, dz \, \frac{d \lambda}{\lambda}\\
        &=\int_{0}^1\int_{\mathbb{R}^n} |a(z)|^2\left| \int_{\mathbb{R}^n}  a(y) \left(\lambda^{n/2}\tilde{\psi}_{\tfrac{1}{\lambda}} \ast \lambda^{n/2}{\psi}_{\tfrac{1}{\lambda}}\right) (z-y) f(y)  \, dy \right|^2  \, dz \, \frac{d \lambda}{\lambda}\\
        & \leq n^2 \|D \tau\|_\infty^2 \int_{0}^1\int_{\mathbb{R}^n} \left| \int_{\mathbb{R}^n}  a(y) \left(\lambda^{n/2}\tilde{\psi}_{\tfrac{1}{\lambda}} \ast \lambda^{n/2}{\psi}_{\tfrac{1}{\lambda}}\right) (z-y) f(y)  \, dy \right|^2  \, dz \, \frac{d \lambda}{\lambda}.
    \end{align*}
    Let $F(y) = a(y)f(y) \in \Lb^2(\mathbb{R}^n)$ and let $\mathcal{F}$ represent taking the Fourier Transform. Then we substitute $F(y)$ for $a(y)f(y)$ in the last line of the inequality above to get 
    \begin{align*}
        & n^2 \|D \tau\|_\infty^2 \int_{0}^1\int_{\mathbb{R}^n} \left| \int_{\mathbb{R}^n}  a(y) \left(\lambda^{n/2}\tilde{\psi}_{\tfrac{1}{\lambda}} \ast \lambda^{n/2}{\psi}_{\tfrac{1}{\lambda}}\right) (z-y) f(y)  \, dy \right|^2  \, dz \, \frac{d \lambda}{\lambda} \\
        &= n^2 \|D \tau\|_\infty^2 \int_{0}^1\int_{\mathbb{R}^n} \left| \int_{\mathbb{R}^n}  \left(\lambda^{n/2}\tilde{\psi}_{\tfrac{1}{\lambda}} \ast \lambda^{n/2}{\psi}_{\tfrac{1}{\lambda}}\right) (z-y) F(y)  \, dy \right|^2  \, dz \, \frac{d \lambda}{\lambda} \\
        &=  n^2 \|D \tau\|_\infty^2 \int_{0}^1\int_{\mathbb{R}^n} \left| \mathcal{F}\left(\lambda^{n/2}\tilde{\psi}_{\tfrac{1}{\lambda}} \ast \lambda^{n/2}{\psi}_{\tfrac{1}{\lambda}}\right) (\omega) \hat{F}(\omega)  \right|^2  \, dz \, \frac{d \lambda}{\lambda} \\
        &= n^2 \|D \tau\|_\infty^2  \int_{\mathbb{R}^n} |\hat{F}(\omega)|^2 \left(\int_{0}^1 |\hat{\psi}(\tfrac{\omega}{\lambda})|^4
\frac{d \lambda}{\lambda}\right) \, d\omega.
    \end{align*}
To finish up the argument, we make a substitution to rewrite 
$$\int_{0}^1 |\hat{\psi}(\tfrac{\omega}{\lambda})|^4
\frac{d \lambda}{\lambda} = \int_{1}^\infty |\hat{\psi}(\lambda \omega)|^4
\frac{d \lambda}{\lambda}.$$
Using our decay assumptions on $\psi$ and its partial derivatives, from Problem 6.1.3 in \cite{grafakos}, we know that
$$|\hat{\psi}(\omega)| \leq  M_\psi\text{min}\{|\omega|, |\omega|^{-2}\}$$
for some constant $M_\psi.$ Now, consider the quantity $\int_{0}^\infty |\hat{\psi}(\lambda \omega)|^4
\frac{d \lambda}{\lambda}.$ Without loss of generality, assume that $|\omega| = 1$ since dilations of $\omega$ do not change the integral. It follows that 
$$\int_{0}^\infty |\hat{\psi}(\lambda \omega)|^4
\frac{d \lambda}{\lambda} \leq M_\psi \int_{0}^1 \lambda^3 d\lambda + M_\psi\int_{1}^\infty \lambda^{-9} d\lambda  < \infty,$$
and we conclude that 
$$ \int_{1}^\infty |\hat{\psi}(\lambda \omega)|^4
\frac{d \lambda}{\lambda}\leq A_\psi$$
for some constant $A_\psi$. To finish up, 
\begin{align*}
n^2 \|D \tau\|_\infty^2  \int_{\mathbb{R}^n} |\hat{F}(\omega)|^2 \left(\int_{0}^1 |\hat{\psi}(\tfrac{\omega}{\lambda})|^4
\frac{d \lambda}{\lambda}\right) \, d\omega &\leq n^2 \|D \tau\|_\infty^2 A_\psi \int_{\mathbb{R}^n} |\hat{F}(\omega)|^2  \, d\omega \\
&\leq n^2 \|D \tau\|_\infty^2 A_\psi \int_{\mathbb{R}^n} |a(z)f(z)|^2  \, dz \\
& \leq n^2 \|D \tau\|_\infty^2 A_\psi \|f\|_2^2.
\end{align*}Thus, we have the desired bound on $\norm[\bigg]{\int_0^{1} \tilde{K}_{\lambda,1}^*\tilde{K}_{\lambda,1} \ \frac{d\lambda}{\lambda}}^{1/2}$.

Substituting everything in yields
	\begin{align*}
	\norm[\bigg]{\int_0^{2^{-\gamma}} K_\lambda^*K_\lambda \ \frac{d\lambda}{\lambda}}^{1/2} &= \norm[\bigg]{\int_0^{2^{-\gamma}} (\tilde{K}_{\lambda,1}+\tilde{K}_{\lambda,2})^*(\tilde{K}_{\lambda,1}+\tilde{K}_{\lambda,2}) \ \frac{d\lambda}{\lambda}}^{1/2} \\
	&= \norm[\bigg]{\int_0^{2^{-\gamma}} (\tilde{K}_{\lambda,1}^*\tilde{K}_{\lambda,1}+\tilde{K}_{\lambda,1}^*\tilde{K}_{\lambda,2}+\tilde{K}_{\lambda,2}^*\tilde{K}_{\lambda,1}+\tilde{K}_{\lambda,2}^*\tilde{K}_{\lambda,2}) \ \frac{d\lambda}{\lambda}}^{1/2} \\
	&\leq \left(\norm[\bigg]{\int_0^{2^{-\gamma}} \tilde{K}_{\lambda,1}^*\tilde{K}_{\lambda,1} \frac{d\lambda}{\lambda}}+\norm[\bigg]{\int_0^{2^{-\gamma}} \tilde{K}_{\lambda,1}^*\tilde{K}_{\lambda,2}+\tilde{K}_{\lambda,2}^*\tilde{K}_{\lambda,1}+\tilde{K}_{\lambda,2}^*\tilde{K}_{\lambda,2} \ \frac{d\lambda}{\lambda}}\right)^{1/2} \\ 
	&\leq \left(\norm[\bigg]{\int_0^{2^{-\gamma}} \tilde{K}_{\lambda,1}^*\tilde{K}_{\lambda,1} \frac{d\lambda}{\lambda}}+\int_0^{2^{-\gamma}} \norm{\tilde{K}_{\lambda,2}}^2\ \frac{d\lambda}{\lambda}+\int_0^{2^{-\gamma}}2\norm{\tilde{K}_{\lambda,1}}\norm{\tilde{K}_{\lambda,2}} \ \frac{d\lambda}{\lambda}\right)^{1/2} \\
	&\leq \norm[\bigg]{\int_0^{2^{-\gamma}} \tilde{K}_{\lambda,1}^*\tilde{K}_{\lambda,1} \frac{d\lambda}{\lambda}}^{1/2}+\left(\int_0^{2^{-\gamma}} \norm{\tilde{K}_{\lambda,2}}^2\ \frac{d\lambda}{\lambda}\right)^{1/2}+\left(\int_0^{2^{-\gamma}}2\norm{\tilde{K}_{\lambda,1}}\norm{\tilde{K}_{\lambda,2}} \ \frac{d\lambda}{\lambda}\right)^{1/2} \\
	&\leq 2C_n\left(\norm{D\tau}_\infty + \norm{\Delta\tau}_\infty\left(\int_0^{2^{-\gamma}} \lambda^2\ \frac{d\lambda}{\lambda}\right)^{1/2}+\norm{D\tau}_\infty^{1/2}\norm{\Delta\tau}_\infty^{1/2} \left(\int_0^{2^{-\gamma}} 2\lambda \ \frac{d\lambda}{\lambda}\right)^{1/2} \right) \\
	&\leq 2C_n\left(\norm{D\tau}_\infty + 2^{-\gamma} \norm{\Delta\tau}_\infty +2^{-\gamma/2}\norm{D\tau}_\infty^{1/2} \norm{\Delta\tau}_\infty^{1/2} \right) \\
	&\leq 4C_n\left(\norm{D\tau}_\infty + 2^{-\gamma} \norm{\Delta\tau}_\infty \right).
	\end{align*}

To bound $P_3$, we decompose $K_\lambda = K_{\lambda,1}+K_{\lambda,2}$, where the kernels defining $K_{\lambda,1}, K_{\lambda,2}$ are
	\begin{align*}
	k_{\lambda,1}(x,u) &= \lambda^n\psi(\lambda (x-u)) - \lambda^{n}\psi(\lambda(I-D\tau(u))(x-u))\det(I-D\tau(u)) \\
	k_{\lambda,2}(x,u) &=\det(I-D\tau(u)) \lambda^{n}\psi(\lambda(I-D\tau(u))(x-u)) - \lambda^{n}\psi(\lambda(x-\tau(x)-u+\tau(u))).
	\end{align*}
	A similar computation to the one for $P_1$ shows that:
	\begin{align*}
	\norm[\bigg]{\int_1^{\infty} K_\lambda^*K_\lambda \ \frac{d\lambda}{\lambda}}^{1/2} &\leq \norm[\bigg]{\int_1^{\infty} K_{\lambda,1}^*K_{\lambda,1} \frac{d\lambda}{\lambda}}^{1/2}+\left(\int_1^{\infty} \norm{K_{\lambda,2}}^2\ \frac{d\lambda}{\lambda}\right)^{1/2}+\left(\int_1^{\infty}2\norm{K_{\lambda,1}}\norm{K_{\lambda,2}} \ \frac{d\lambda}{\lambda}\right)^{1/2} 
	\end{align*}
	Letting $Q_j = K_{2^j,1}^*K_{2^j,1}$, it is shown in \cite{mallat2012group} that:
		 \begin{align*}
	\norm{K_{\lambda,1}} &\leq C_n\norm{D\tau}_\infty \\
	\norm{K_{\lambda,2}} &\leq \min\{\lambda^{-n}\norm{D^2\tau}_\infty, \norm{D\tau}_\infty \} \\
	\norm{Q_{j} Q_{\ell}} &\leq C_n^2 2^{-|j-\ell|}(\norm{D\tau}_\infty+\norm{D^2\tau}_\infty)^4
	\end{align*}
	so that
	\begin{align*}
	 \norm[\bigg]{\int_1^{\infty} K_{\lambda,1}^*K_{\lambda,1} \frac{d\lambda}{\lambda}}^{1/2}
	 &= \norm[\bigg]{\int_0^{\infty} K_{2^j,1}^*K_{2^j,1} \log(2)\ dj }^{1/2} \\
	 &= \sqrt{\log(2)}\ \norm[\bigg]{\int_0^{\infty} Q_j \ dj }^{1/2}. 
	\end{align*}
	We now apply a continuous version of Cotlar's Lemma (see Ch. 7 of \cite{stein2016harmonic}, Sec. 5.5 for the continuous extension). We define: 
	\begin{align*}
	\beta(j,\ell) &= \begin{cases} C_n 2^{-|j-\ell|/2}(\norm{D\tau}_\infty+\norm{D^2\tau}_\infty)^2 & j \geq 0 \text{ and } \ell \geq 0 \\
	0 & \text{otherwise} \end{cases}.
	\end{align*}
	Defining $Q_j=0$ for $j<0$, we have $\norm{Q_j^*Q_\ell} \leq \beta(j,\ell)^2$ and $\norm{Q_jQ_\ell^*} \leq \beta(j,\ell)^2$ for all $j,\ell$. Thus by Cotlar's Lemma:
	\begin{align*}
	\norm[\bigg]{\int_{\mathbb{R}} Q_j \ dj } &\leq \sup_{j \in \mathbb{R}} \int_{\mathbb{R}} \beta(j,\ell)\ d\ell, \\
	\norm[\bigg]{\int_{0}^\infty Q_j \ dj } &\leq \sup_{j \geq 0} \int_0^\infty \beta(j,\ell)\ d\ell \\
	&\leq C_n(\norm{D\tau}_\infty+\norm{H\tau}_\infty)^2\left( \sup_{j \geq 0} \int_0^\infty 2^{-|j-\ell|/2} \ d\ell \right).
	\end{align*}
	Now observing that with the change of variable $\lambda_1 =2^j, \lambda_2=2^\ell$, we have $2^{-|j-\ell|/2} = \frac{\lambda_1}{\lambda_2} \wedge \frac{\lambda_2}{\lambda_1}$, we obtain:
	\begin{align*}
	\sup_{j \geq 0} \int_0^\infty 2^{-|j-\ell|/2} \ d\ell &= \sup_{\lambda_1 \geq 1} \int_1^\infty \frac{(\lambda_1 \wedge \lambda_2)}{\sqrt{\lambda_1\lambda_2}} \ \frac{d\lambda_2}{\ln(2) \lambda_2} \\
	&= \frac{1}{\ln(2)} \sup_{\lambda_1 \geq 1} \left(\int_1^{\lambda_1} \frac{1}{\sqrt{\lambda_1\lambda_2}}\ d\lambda_2 + \int_{\lambda_1}^\infty  \frac{\sqrt{\lambda_1}}{\lambda_2^{3/2}}\ d\lambda_2 \right) \\
	&= \frac{1}{\ln(2)} \sup_{\lambda_1 \geq 1} \left(\frac{1}{\sqrt{\lambda_1}}(2\sqrt{\lambda_1}-2)+\sqrt{\lambda_1}\left(\frac{2}{\sqrt{\lambda_1}}\right)\right) \\
	&= \frac{1}{\ln(2)} \sup_{\lambda_1 \geq 1} \left(4-\frac{2}{\sqrt{\lambda_1}}\right) \\
	&= \frac{4}{\ln(2)}
	\end{align*}
	and conclude that
	\begin{align*}
	\norm[\bigg]{\int_1^{\infty} K_{\lambda,1}^*K_{\lambda,1} \frac{d\lambda}{\lambda}}^{1/2} &\leq 3 C_n(\norm{D\tau}_\infty+\norm{H\tau}_\infty).
	\end{align*}	
Thus we have:
\begin{align*}
\norm[\bigg]{\int_1^{\infty} K_\lambda^*K_\lambda \ \frac{d\lambda}{\lambda}}^{1/2} &\leq \norm[\bigg]{\int_1^{\infty} K_{\lambda,1}^*K_{\lambda,1} \frac{d\lambda}{\lambda}}^{1/2}+\left(\int_1^{\infty} \norm{K_{\lambda,2}}^2\ \frac{d\lambda}{\lambda}\right)^{1/2}+\left(\int_1^{\infty}2\norm{K_{\lambda,1}}\norm{K_{\lambda,2}} \ \frac{d\lambda}{\lambda}\right)^{1/2}.
\end{align*}
Now we see that there exists a constant $C_n$ such that 
\begin{align*}
\norm[\bigg]{\int_1^{\infty} K_{\lambda,1}^*K_{\lambda,1} \frac{d\lambda}{\lambda}}^{1/2} & \leq C_n(\norm{D\tau}_\infty+\norm{D^2\tau}_\infty) \\
\left(\int_1^{\infty} \norm{K_{\lambda,2}}^2\ \frac{d\lambda}{\lambda}\right)^{1/2} & \leq C_n\norm{D^2\tau}_\infty\left(\int_1^{\infty} \lambda^{-2n}\ \frac{d\lambda}{\lambda}\right)^{1/2} \\
\left(\int_1^{\infty}2\norm{K_{\lambda,1}}\norm{K_{\lambda,2}} \ \frac{d\lambda}{\lambda}\right)^{1/2}   & \leq C_n\norm{D\tau}_\infty^{1/2}\norm{D^2\tau}_\infty^{1/2} \left(\int_1^{\infty}2\lambda^{-n} \ \frac{d\lambda}{\lambda}\right)^{1/2}.  
\end{align*}
and
\begin{align*}
\norm[\bigg]{\int_1^{\infty} K_\lambda^*K_\lambda \ \frac{d\lambda}{\lambda}}^{1/2} &\leq C_n\left(\norm{D\tau}_\infty+\frac{1}{2n}\norm{D^2\tau}_\infty + \frac{2}{n} \norm{D\tau}_\infty^{1/2}\norm{D^2\tau}_\infty^{1/2}\right) \\
&\leq C_n \left(\norm{D\tau}_\infty+\frac{1}{2n}\norm{D^2\tau}_\infty + \frac{1}{n} \norm{D\tau}_\infty + \frac{1}{n}\norm{D^2\tau}_\infty\right) \\
&\leq 2C_n(\norm{D\tau}_\infty+\norm{D^2\tau}_\infty).
\end{align*}

Finally, we bound $P_2$. Note that in the previous section it was observed (shown in \cite{mallat2012group}) that	
\begin{align*}
\norm{K_{\lambda,1}} &\leq C_n\norm{D\tau}_\infty \\
\norm{K_{\lambda,2}} &\leq \min\{\lambda^{-n}\norm{D^2\tau}_\infty, \norm{D\tau}_\infty \}.
\end{align*}
The above two inequalities imply
\begin{align*}
\norm{K_\lambda} &=\norm{K_{\lambda,1}+K_{\lambda,2}} \leq \norm{K_{\lambda,1}} + \norm{K_{\lambda,2}} \leq 2C_n\norm{D\tau}_\infty
\end{align*}	
so that
\begin{align*}
\norm[\bigg]{\int_{2^{-\gamma}}^1 K_\lambda^*K_\lambda \ \frac{d\lambda}{\lambda}}^{1/2} &\leq \left( \int_{2^{-\gamma}}^1 \norm{K_\lambda}^2 \ \frac{d\lambda}{\lambda} \right)^{1/2} \\
&\leq 2C_n\norm{D\tau}_\infty \left( \int_{2^{-\gamma}}^1 \ \frac{d\lambda}{\lambda} \right)^{1/2} \\
&\leq 2C_n\norm{D\tau}_\infty \left( -\ln(2^{-\gamma}) \right)^{1/2} \\
&\leq 2C_n\gamma^{1/2}\norm{D\tau}_\infty.
\end{align*} 
Putting everything together and since $\gamma \geq 1$, we obtain:
\begin{align*}
\norm{[\Wc,L_\tau]} &\leq 2(P_1+P_2+P_3) \\
&\leq 4C_n\left(\norm{D\tau}_\infty + 2^{-\gamma} \norm{\Delta\tau}_\infty \right) + 2C_n\gamma^{1/2}\norm{D\tau}_\infty + 3C_n(\norm{D\tau}_\infty+\norm{D^2\tau}_\infty) \\
&\leq \tilde{C}_n\left(\gamma\norm{D\tau}_\infty + 2^{-\gamma} \norm{\Delta\tau}_\infty +\norm{D^2\tau}_\infty\right). 
\end{align*}
Choosing $\gamma = \left(\log\frac{\norm{\Delta\tau}_\infty}{\norm{D\tau}_\infty}\right) \vee 1$ gives
\begin{align*}
\norm{[\Wc,L_\tau]} \leq \tilde{C}_n \left(\left(\log\frac{\norm{\Delta\tau}_\infty}{\norm{D\tau}_\infty} \vee 1\right)\norm{D\tau}_\infty + \norm{D^2\tau}_\infty\right),
\end{align*}
and the lemma is proved.
\end{proof}

\begin{theorem}
	\label{prop:genstabcont}
	Assume $\psi$ and its first and second order derivatives have decay in $O((1+|x|)^{-n-3})$ and $\int_{\mathbb{R}^n} \psi(x)\ dx=0$. Then for every $\tau \in C^2(\mathbb{R}^n)$ with $\norm{D\tau}_{\infty} \leq \frac{1}{2n}$, there exists $C_{m,n}>0$ and $\hat{C}_{m,n} > 0$ such that
	\begin{align*}
	\norm{S_{\text{cont},2}^mf - S_{\text{cont},2}^mL_\tau f}_{\Lb^2(\mathbb{R}_+^m)}^2 &\leq C_{m,n}\left(\norm{D\tau}_{\infty}^2+\left(\norm{D\tau}_{\infty}\left(\log\frac{\norm{\Delta\tau}_\infty}{\norm{D\tau}_\infty} \vee 1\right)+\norm{D^2\tau}_\infty\right)^2\right)\norm{f}_2^2.
	\end{align*}
    and
 \begin{align*}
 \norm{S_{\text{dyad},2}^mf - S_{\text{dyad},2}^mL_\tau f}_{\ellb^2(\mathbb{Z}^m)}^2 &\leq \hat{C}_{m,n}\left(\norm{D\tau}_{\infty}^2+\left(\norm{D\tau}_{\infty}\left(\log\frac{\norm{\Delta\tau}_\infty}{\norm{D\tau}_\infty} \vee 1\right)+\norm{D^2\tau}_\infty\right)^2\right)\norm{f}_2^2.
 \end{align*}
\end{theorem}
\begin{proof}
The proof is only provided for the continuous case. We have the following bound for some $C_m$:
\begin{align*}
\norm{S_{\text{cont},2}^mf - S_{\text{cont},2}^mL_\tau f}_{\Lb^2(\mathbb{R}_+^m)} &\leq \|A_2M\Wc V_{m-1}f - A_2ML_\tau \Wc V_{m-1}f \|_{\Lb^2 (\R_+^m)} + \|A_2M[\Wc V_{m-1}, L_\tau]f\|_{\Lb^2 (\R_+^m)} \\
& \leq \|A_2M\Wc V_{m-1}f - A_2ML_\tau \Wc V_{m-1}f \|_{\Lb^2 (\R_+^m)} + C_m^2 \norm{[\Wc,L_\tau]}_{\Lb^2(\R_{+}^m \times \R^n) \to  \Lb^2(\R^n)}^2\|f\|_2^2.
\end{align*} 
For the first term, we can mimic the dilation argument to get
$$|A_2M\Wc V_{m-1}f - A_2ML_\tau \Wc V_{m-1}f| = \left| \|g\|_2 - \|L_\tau g\|_2\right|.$$
The difference is the term with the diffeomorphism. Let $y = \gamma(x) = x - \tau(x)$. Then it follows that $\gamma^{-1}(y) = x$ and change of variables implies that
\begin{align*}
\norm{L_\tau f}_2^2 = \int_{\mathbb{R}^n} |f(x-\tau(x))|^2\ dx =  \int_{\mathbb{R}^n} |f(y)|^2 \, \frac{dy}{|\det(I-D\tau(\gamma^{-1}(y)))|}.
\end{align*}
We also have
$$1 - n\|D\tau\|_\infty \leq |\det(I-D\tau(\gamma^{-1}(y)))| \leq 1 + n\|D\tau\|_\infty.$$
Thus, we obtain
\begin{align*}
\frac{1}{1+n\norm{D\tau}_{\infty}}\int_{\mathbb{R}^n} |f(y)|^2 \ dy \leq \norm{L_\tau f}_2^2 &\leq \frac{1}{1-n\norm{D\tau}_{\infty}}\int_{\mathbb{R}^n} |f(y)|^2 \ dy, \\
\frac{1}{1+n\norm{D\tau}_{\infty}}\norm{f}_2^2 \leq \norm{L_\tau f}_2^2 &\leq \frac{1}{1-n\norm{D\tau}_{\infty}}\norm{f}_2^2.
\end{align*}
Since we have a bound on $\|D\tau\|_\infty$, we see that 
$$\frac{1}{1+n\norm{D\tau}_{\infty}} = \frac{1- n\|D\tau\|_\infty}{1-n^2 \|D\tau\|_\infty^2} \geq 1- n\|D\tau\|_\infty$$
since $1 > 1-n^2 \|D\tau\|_\infty^2 > 0.$ Similarly, 
$$\frac{1}{1-n\norm{D\tau}_{\infty}} = \frac{1+ 2n\|D\tau\|_\infty}{1 + n \|D\tau\|_\infty -2n^2 \|D\tau\|_\infty^2}$$
and
$$1 + n \|D\tau\|_\infty -2n^2 \|D\tau\|_\infty^2  \geq 1 + n \|D\tau\|_\infty -\frac{2n^2}{2n} \|D\tau\|_\infty = 1$$
since $\|D\tau\|_\infty \leq \tfrac{1}{2n}$. It follows that $\frac{1}{1-n\norm{D\tau}_{\infty}} \leq 1+ 2n\|D\tau\|_\infty$ and
$$(1 - n\|D\tau\|_\infty)^{1/2}\norm{f}_2 \leq \norm{L_\tau f}_2 \leq (1+ 2n\|D\tau\|_\infty)^{1/2}\norm{f}_2.$$
Since $1 - n\|D\tau\|_\infty < 1$ and $1+ 2n\|D\tau\|_\infty > 1$, Use the lower bound on $\norm{L_\tau f}_2$ to get
\begin{align*}
\|f\|_2 - \norm{L_\tau f}_2 & = \|f\|_2 \left(1 - (1 - n\|D\tau\|_\infty)^{1/2}\right) \\
                            &\leq \|f\|_2 \left(1 - (1 - n\|D\tau\|_\infty)\right) \\
                            &= n\|D\tau\|_\infty \|f\|_2.
\end{align*} and the upper bound to get
\begin{align*}
 \norm{L_\tau f}_2 - \|f\|_2 & = \|f\|_2 \left((1 + 2n\|D\tau\|_\infty)^{1/2} - 1\right) \\
                            &\leq \|f\|_2 \left((1 + 2n\|D\tau\|_\infty) - 1\right) \\
                            &= 2n\|D\tau\|_\infty \|f\|_2.
\end{align*} 
Finally, we have
$$\left|\|f\|_2 - \norm{L_\tau f}_2\right| \leq 2n\|D\tau\|_\infty\|f\|_2$$
for any $f \in {\Lb}^2(\mathbb{R}^n).$ Now we mimic the argument given for dilation stability to get 
$$\|A_2M\Wc V_{m-1}f - A_2ML_\tau \Wc V_{m-1}f\|_{\Lb^2 (\R_+^m)}^2 \leq C\|D\tau\|_\infty^2 \|f\|_2^2$$
for some constant $C.$ For the second term, we have
$$C_m^2 \norm{[\Wc,L_\tau]}_{\Lb^2(\R_{+}^m \times \R^n) \to  \Lb^2(\R^n)}^2\|f\|_2^2 \leq  C'\left(\norm{D\tau}_{\infty}\left(\log\frac{\norm{\Delta\tau}_\infty}{\norm{D\tau}_\infty} \vee 1\right)+\norm{D^2\tau}_\infty\right)^2 \|f\|_2^2$$
for some constant $C'$. We now choose $C_{n,m} = \max\{C', C\}$ to get the desired bound.
\end{proof}

\subsection{Stability to Diffeomorphisms When $1 < q < 2$}
\begin{lemma} \label{lemma: 2 norm defined} 
Let $\gamma(z)  = z - \tau(z)$, $g(z) = f(\gamma(z))$, and
$$K_\lambda(x,z) = \det(D\gamma(z)) \psi_\lambda(\gamma(x)- \gamma(z)) - \psi_\lambda(x-z).$$
Additionally, define
$$T_\lambda g(x) = \int_{\mathbb{R}^n} g(z) K_\lambda(x,z)\, dz$$
and consider $Tg: \mathbb{R}^n \to \Lb^2(\mathbb{R}_{+}, \tfrac{d\lambda}{\lambda^{n+1}})$ defined by $Tg(x) = (T_\lambda g(x))_{\lambda \in \mathbb{R}_{+}}.$ Then for the Banach space $\mathcal{X} =  \Lb^2(\mathbb{R}_+, \tfrac{d\lambda}{\lambda^{n+1}})$,
$$\|Tg\|_{\Lb^2_\mathcal{X}(\mathbb{R}^n)}^2 \leq  C_{n,m} \left(\|D\tau\|_\infty \left(\log\frac{\|\Delta \tau\|_\infty}{\|D\tau\|_\infty}\vee 1\right) + \|D^2\tau\|_\infty \right)^2 \|f\|_2$$
for some constant $C_{n,m} > 0$.
\end{lemma}
\begin{proof}
Notice that 
\begin{align*}
\|Tg\|_{\Lb^2_X(\mathbb{R}^n)}^2 &= \int_{\mathbb{R}^n}\int_{0}^\infty |T_\lambda g(x)|^2 \frac{d \lambda}{\lambda^{n+1}} dx\\
&= \int_{\mathbb{R}^n}\int_{0}^\infty \left|\int_{\mathbb{R}^n} K_\lambda(x,z) g(z)\, dz\right|^2 \frac{d \lambda}{\lambda^{n+1}} dx\\
&= \int_{\mathbb{R}^n}\int_{0}^\infty \left|\int_{\mathbb{R}^n} f(\gamma(z)) [\det(D\gamma(z)) \psi_\lambda(\gamma(x) - \gamma(z)) - \psi_\lambda(x-z)]\, dz\right|^2 \frac{d \lambda}{\lambda^{n+1}} dx \\
 &= \int_{\mathbb{R}^n}\int_{0}^\infty
 \left|\int_{\mathbb{R}^n} \det(D\gamma(z))f(\gamma(z))\psi_\lambda(\gamma(x) - \gamma(z)) \, dz  -\int_{\mathbb{R}^n}f(\gamma(z)) \psi_\lambda(x-z)\, dz\right|^2 \frac{d \lambda}{\lambda^{n+1}} dx.
\end{align*}

Using the change of variables $u = \gamma(z)$, we get
\begin{align*}
\|Tg\|_{L^2_X(\mathbb{R}^n)}^2 &= \int_{\mathbb{R}^n}\int_{0}^\infty \left|L_\tau(f \ast \psi_\lambda)(x)- (L_\tau f \ast \psi_\lambda)(x)\right|^2 \frac{d \lambda}{\lambda^{n+1}} \, dx\\
&= \int_{\mathbb{R}^n}\int_{0}^\infty \left|[\Wc_\lambda, L_\tau] f(x)\right|^2 \frac{d \lambda}{\lambda^{n+1}} \, dx\\
&= \int_{0}^\infty  \int_{\mathbb{R}^n}\left|[\Wc_\lambda, L_\tau] f(x)\right|^2\, dx \, \frac{d \lambda}{\lambda^{n+1}}\\
&= \int_{0}^\infty  \left\|[\Wc_\lambda, L_\tau] f\right\|^2_2\, \frac{d \lambda}{\lambda^{n+1}}\\
&= \norm{[\Wc,L_\tau]f}_{\Lb^2(\R_{+} \times \R^n)}^2\\
& \leq C_{n,m} \left(\|D\tau\|_\infty \left(\log\frac{\|\Delta \tau\|_\infty}{\|D\tau\|_\infty} \vee 1\right) + \|D^2\tau\|_\infty \right) \|f\|_2^2,
\end{align*} where the last inequality follows from the $q = 2$ case.
\end{proof}

\begin{lemma} [\cite{vectorvaluedinequalities}, Marcinkiewicz Interpolation] 
Let $\mathcal{A}$ and $\mathcal{B}$ be Banach spaces and let $T:\mathcal{A} \to \mathcal{B}$ be a quasilinear operator defined on $\Lb^{p_0}_\mathcal{A}(\mathbb{R}^n)$ and $\Lb^{p_1}_\mathcal{A}(\mathbb{R}^n)$ with $0 < p_0 < p_1$. Furthermore, if $T$ satisfies
$$\|Tf\|_{\Lb^{p_i,\infty}_{\mathcal{B}}(\mathbb{R}^n)} \leq M_i\|f\|_{\Lb_\mathcal{A}^{p_i}(\mathbb{R}^n)}$$ for $i = 0, 1$, then for all $p \in (p_0, p_1)$,
$$\|Tf\|_{\Lb^p_\mathcal{B}(\mathbb{R}^n)} \leq N_p \|f\|_{\Lb_\mathcal{A}^{p}(\mathbb{R}^n)},$$
where $N_p$ only depends on $M_0$, $M_1$, and $p$.
\end{lemma}

\begin{remark}
Like with the scalar valued estimate, it can be shown that $N_p = \eta M_0^\delta M_1^{1- \delta}$, where 
$$
\delta = 
\begin{dcases}
\frac{p_0(p_1-p)}{p(p_1-p_0)} &\quad p_1 < \infty, \\
\frac{p_0}{p} &\quad  p_1 = \infty
\end{dcases}
$$
and 
$$
\eta = 
\begin{dcases}
2\left(\frac{p(p_1-p_0)}{(p-p_0)(p_1-p)}\right)^{1/p} &\quad p_1 < \infty, \\
2\left(\frac{p_0}{p-p_0}\right)^{1/p} &\quad  p_1 =  \infty.
\end{dcases}
$$ 
\end{remark}

\begin{lemma} \label{lemma: interpolation}
Let $T$ be the operator defined in Lemma \ref{lemma: 2 norm defined}. Let $q \in (1,2)$ and $r \in (1,q)$. Then $T$ satisfies 
$$\|Tg\|_{\Lb^{r,\infty}_{\mathcal{X}}(\mathbb{R}^n)} \leq M_r\|f\|_{\Lb^{r}(\mathbb{R}^n)}$$
for some constant $M_r>0$, which is independent of $\|D\tau\|_\infty$ and $\|D^2 \tau\|_\infty$. Furthermore, $T$ also satisfies 
$$\|Tg\|_{\Lb^{2,\infty}_{\mathcal{X}}(\mathbb{R}^n)}^2 \leq \tilde{C}_{n} \left(\|D\tau\|_\infty \left(\log\frac{\|\Delta \tau\|_\infty}{\|D\tau\|_\infty} \vee 1\right) + \|D^2\tau\|_\infty \right)^2 \|f\|_{\Lb^{2}(\mathbb{R}^n)}^2$$ 
for some constant $\tilde{C}_{n} > 0$.
\end{lemma}
\begin{proof} The second inequality obviously follows from strong boundedness of the operator, so we will omit the proof. For the first inequality, the norm satisfies 
\begin{align*}
\|Tg(x)\|_{\mathcal{X}}^2  &= \int_{0}^\infty
 \left|\int_{\mathbb{R}^n} \det(D\gamma(z))f(\gamma(z))\psi_\lambda(\gamma(x) - \gamma(z)) \, dz  -\int_{\mathbb{R}^n}f(\gamma(z)) \psi_\lambda(x-z)\, dz\right|^2 \frac{d \lambda}{\lambda^{n+1}}\\
 &= \int_{0}^\infty
 \left|\int_{\mathbb{R}^n} f(z)\psi_\lambda(\gamma(x) - z) \, dz  -\int_{\mathbb{R}^n}f(\gamma(z)) \psi_\lambda(x-z)\, dz\right|^2 \frac{d \lambda}{\lambda^{n+1}}\\
 & \leq 4 \int_{0}^\infty
 \left|\int_{\mathbb{R}^n} f(z)\psi_\lambda(\gamma(x) - z) \, dz\right|^2\frac{d \lambda}{\lambda^2}  + 4\int_{0}^\infty\left|\int_{\mathbb{R}^n}f(\gamma(z)) \psi_\lambda(x-z)\, dz\right|^2\frac{d \lambda}{\lambda^{n+1}}\\
 &= 4|(Gf)(\gamma(x))|^2 + 4|GL_\tau f(x)|^2.
\end{align*}  

We see 
$$\|Tg(x)\|_{\mathcal{X}} \leq \sqrt{4|(Gf)(\gamma(x))|^2 + 4|GL_\tau f(x)|^2} \leq 2|(Gf)(\gamma(x))| + 2|GL_\tau f(x)|.$$
For $\delta > 0$, Chebyshev's inequality implies that there exists $A_r$ such that
\begin{align*}
m\{\|Tg(x)\|_{\mathcal{X}} > \delta\} &\leq m\{2|(Gf)(\gamma(x))| + 2|GL_\tau f(x)| > \delta\} \\
& \leq \frac{A_r}{\delta^{r}} (\|(Gf)(\gamma(\cdot))\|_{\Lb^r(\mathbb{R}^n)}^r + \|GL_\tau f\|_{\Lb^r(\mathbb{R}^n)}^r).
\end{align*}
We want to now ensure that
$\|(Gf)(\gamma(\cdot))\|_{\Lb^r(\mathbb{R}^n)}^r$ can be bounded above by a constant multiple of $\|Gf\|_{\Lb^r(\mathbb{R}^n)}^r.$ Since $\gamma$ is a diffeomorphism, we can use change of variables to get 
\begin{align*}
\|(Gf)(\gamma(\cdot))\|_{\Lb^r(\mathbb{R}^n)}^r &= \int_{\mathbb{R}^n} |Gf(\gamma(x))|^r \, dx \\
&= \int_{\mathbb{R}^n} |Gf(u)|^r \, \frac{du}{\det\left[(D\gamma)(\gamma^{-1}(u))\right]} \\
&\leq  2 \int_{\mathbb{R}^n} |Gf(x)|^r  \, dx\\
&=  2 \|Gf\|_{\Lb^r(\mathbb{R}^n)}^r.
\end{align*}
By Theorem \ref{thm: LP g-function Lp bounded}, we get
\begin{align*}
    \|GL_\tau f\|_{\Lb^r(\mathbb{R}^n)}^r \leq C_r\|L_\tau f\|_{\Lb^r(\mathbb{R}^n)}^r \leq 2C_r \|f\|_{\Lb^r(\mathbb{R}^n)}^r
\end{align*} for some constant $C_r$ dependent on $r$.
Thus, we have 
$$m\{\|Tg(x)\|_{\mathcal{X}} > \delta\}^{1/r} \leq \frac{M_r}{\delta} \|f\|_{\Lb^r(\mathbb{R}^n)}$$
for some constant $M_r > 0$.
\end{proof}

\begin{lemma}\label{lemma: interpolate} Fix $r = \frac{1+q}{2}$ so that $r \in (1,q)$. For some constant $C_{n,q} > 0$, the operator $T$ defined in Lemma \ref{lemma: 2 norm defined} satisfies the estimate 
$$\|Tg\|_{\Lb^q_\mathcal{X}(\mathbb{R}^n)}^q \leq C_{n,q} \eta^q M_r^{q\delta}  \left(\|D\tau\|_\infty \left(\log\frac{\|\Delta \tau\|_\infty}{\|D\tau\|_\infty} \vee 1\right) + \|D^2\tau\|_\infty \right)^{q(1-\delta)} \|f\|_q^q,$$
where $\eta$ and $\delta$ come from interpolation, and $M_r$ comes from the constant for weak boundedness in Lemma \ref{lemma: interpolation}.
\end{lemma}
\begin{proof} Since $T$ is an integral operator, it is clear that is quasilinear. Using the $\Lb^r(\mathbb{R}^n)$ and $\Lb^2(\mathbb{R}^n)$ estimates from the previous Lemma, we interpolate using Marcinkiewicz since $\|g\|_r \leq 2 \|f\|_r \leq 4 \|g\|_r$. 
\end{proof}

\begin{theorem}
Let $1 < q < 2$. Assume $\psi$ and its first and second order derivatives have decay in $O((1+|x|)^{-n-3})$, and $\int_{\mathbb{R}^n} \psi(x)\ dx=0$. Then for every $\tau \in C^2(\mathbb{R}^n)$ with $\norm{D\tau}_{\infty} < \frac{1}{2n}$, there exists $C_{n,q} > 0$ such that 
$$\|S_{\text{cont},q}f - S_{\text{cont},q} L_\tau f \|_{\Lb^2(\R_{+})}^q \leq C_{n,q} \left[\|D\tau\|_\infty^q + \eta^q M_r^{q\delta}  \left(\|D\tau\|_\infty \left(\log\frac{\|\Delta \tau\|_\infty}{\|D\tau\|_\infty} \vee 1\right) + \|D^2\tau\|_\infty \right)^{q(1-\delta)} \right]\|f\|_q^q.$$
\end{theorem}
\begin{proof}
We use the same notation as Theorem \ref{dilate m layers wavelet}. Using a nearly identical argument to Corollary \ref{corollary: m layer dilation stability for q != 2}, we get
\begin{align*}
\|S_{\text{cont},q}f - S_{\text{cont},q} L_\tau f \|_{\Lb^2(\R_{+})} &= \| A_q M \Wc f - A_qM\Wc L_\tau f\|_{\Lb^2(\R_{+})} \\
&= \| A_q M \Wc f - A_qML_\tau \Wc f + A_qML_\tau Wf - A_qM\Wc L_\tau f\|_{\Lb^2(\R_{+})}\\
& \leq \| A_q M \Wc f - A_qML_\tau \Wc f\|_{\Lb^2(\R_{+})} + \|A_qML_\tau \Wc f - A_qM\Wc L_\tau f\|_{\Lb^2(\R_{+})}\\
&\leq \|(A_q M - A_qML_\tau)\Wc f \|_{\Lb^2(\R_{+})} + \|A_qM[\Wc, L_\tau]f\|_{\Lb^2(\R_{+})}.
\end{align*} 
The first term, $\|(A_q M - A_qML_\tau)\Wc f \|_{\Lb^2(\R_{+})}$, can be bounded using an argument identical to the $q = 2$ case. In particular, we can prove that
$$(1 - n\|D\tau\|_\infty) \norm{f}_q \leq (1 - n\|D\tau\|_\infty)^{1/q}\norm{f}_q \leq \norm{L_\tau f}_q$$
and
$$\norm{L_\tau f}_q \leq (1+ 2n\|D\tau\|_\infty)^{1/q}\norm{f}_q \leq (1+ 2n\|D\tau\|_\infty)\norm{f}_q,$$
which means
$$\|(A_q M - A_qML_\tau)\Wc f \|_{\Lb^2(\R_{+})}^q \leq C \|D\tau\|_\infty^q \norm{f}_q^q.$$
For the other term,
\begin{align*}
 \|A_qM[\Wc, L_\tau]f\|_{\Lb^2(\R_{+})}^q 
 &= \left(\int_0^\infty \left[\int_{\mathbb{R}^n}|(L_\tau f \ast \psi_\lambda)(x) - L_\tau(f \ast \psi_\lambda)(x)|^q\, dx\right]^{2/q} \frac{d \lambda}{\lambda^{n+1}}\right)^{q/2}.
 \end{align*}
 
Now, expand convolution and then use change of variables to get
\begin{align*}
 &\|A_qM[\Wc, L_\tau]f\|_{\Lb^2(\R_{+})}^q \\ &= \left(\int_0^\infty \left[\int_{\mathbb{R}^n}\left|\int_{\mathbb{R}^n} f(\gamma(z))(\det(D\gamma(z)) \psi_\lambda(\gamma(x) - \gamma(z)) - \psi_\lambda(x-z)) \, dz \right|^q\, dx\right]^{2/q} \frac{d \lambda}{\lambda^{n+1}}\right)^{q/2} \\
 &= \left(\int_0^\infty\left[\int_{\mathbb{R}^n}\left|\int_{\mathbb{R}^n}g(z)K_\lambda(x,z) \, dz \right|^q\, dx\right]^{2/q} \frac{d \lambda}{\lambda^{n+1}}\right)^{q/2} \\
 &= \left(\int_0^\infty\left[\int_{\mathbb{R}^n}\left|T_\lambda g(x) \right|^q\, dx\right]^{2/q}\frac{d \lambda}{\lambda^{n+1}}\right)^{q/2} \\
 & \leq \int_{\mathbb{R}^n}\left[\int_0^\infty\left|T_\lambda g(x) \right|^{q}\, \frac{d \lambda}{\lambda^{n+1}}\right]^{q/2}\, dx  \\
 &=\int_{\mathbb{R}^n} \left[\int_0^\infty\left|T_\lambda g(x) \right|^{2}\, \frac{d \lambda}{\lambda^{n+1}}\right]^{q/2}\, dx \\
 &= \int_{\mathbb{R}^n} \|Tg(x)\|^q_{\Lb^2\left(\mathbb{R}^{+}, \tfrac{d \lambda}{\lambda^{n+1}}\right)} \, dx\\
 &= \|Tg\|^q_{\Lb^q_\mathcal{X}(\mathbb{R}^n)}\\
 & \leq C_n \eta^q M_r^{q\delta}  \left(\|D\tau\|_\infty \left(\log\frac{\|\Delta \tau\|_\infty}{\|D\tau\|_\infty} \vee 1\right) + \|D^2\tau\|_\infty \right)^{q(1-\delta)} \|f\|_q^q.
 \end{align*} Thus, the proof is complete.
 \end{proof}
 
\begin{corollary} Let $1 < q < 2$ . Assume $\psi$ and its first and second order derivatives have decay in $O((1+|x|)^{-n-3})$, and $\int_{\mathbb{R}^n} \psi(x)\ dx=0$. Then for every $\tau \in C^2(\mathbb{R}^n)$ with $\norm{D\tau}_{\infty} < \frac{1}{2n}$, there exist constants $C_{n,m}, \hat{C}_{n,m} > 0$ such that 
$$\|S_{\text{cont},q}^m f - S_{\text{cont},q}^m L_\tau f \|_{\Lb^2(\R_{+}^m)}^q \leq   C_{n,m} \left[\|D\tau\|_\infty^q + \eta^q M_r^{q\delta}  \left(\|D\tau\|_\infty \left(\log\frac{\|\Delta \tau\|_\infty}{\|D\tau\|_\infty} \vee 1\right) + \|D^2\tau\|_\infty \right)^{q(1-\delta)}\right] \|f\|_q^q$$
and
$$\|S_{\text{dyad},q}^m f - S_{\text{dyad},q}^m L_\tau f \|_{\ell^2(\Z^m)}^q \leq   \hat{C}_{n,m} \left[\|D\tau\|_\infty^q + \eta^q M_r^{q\delta}  \left(\|D\tau\|_\infty \left(\log\frac{\|\Delta \tau\|_\infty}{\|D\tau\|_\infty} \vee 1\right) + \|D^2\tau\|_\infty \right)^{q(1-\delta)}\right] \|f\|_q^q.$$
\end{corollary}

\begin{remark}
This bound is not exactly the same as the definition for stability to diffeomorphisms in \cite{mallat2012group}, but the idea is similar. Since $r$ is fixed, so is $\delta$. It is easy to confirm that $\delta = \tfrac{1}{1+q} \in \left(\tfrac{1}{3},\tfrac{1}{2}\right)$ when using Marcinkiewicz interpolation in Lemma \ref{lemma: interpolate}, so 
$$C_{n,q} \eta^q M_r^{q\delta}  \left(\|D\tau\|_\infty \left(\log\frac{\|\Delta \tau\|_\infty}{\|D\tau\|_\infty} \vee 1\right) + \|D^2\tau\|_\infty \right)^{q(1-\delta)} \to 0$$
when $\|D\tau\|_{\infty} \to 0$ and $\|D^2\tau\|_{\infty} \to 0$.
\end{remark} 

\section{Equivariance and Invariance to Rotations}
We now consider adding group actions to our scattering transform and prove invariance to rotations. Let $\text{SO}(n)$ be the group of $n \times n$ rotation matrices. Since $\text{SO}(n)$ is a compact Lie group, we can define a Haar measure, say $\mu$, with $\mu(\text{SO}(n)) < \infty$. We say that $f \in \Lb^2(\text{SO}(n))$ if and only if $f$ is $\mu$-measurable and $\int_{\text{SO}(n)} |f(r)|^2 \, d\mu(r) < \infty$.

\subsection{Rotation Equivariant Representations}
Let $\psi:\mathbb{R}^n \to \mathbb{R}$ be a wavelet. Define
\begin{equation*}
    \psi_{\lambda, R}(x) = \lambda^{-n/2} \psi(\lambda^{-1}R^{-1} x),
\end{equation*}
where $R \in \text{SO}(n)$ is a $n \times n$ rotation matrix. The continuous and dyadic wavelet transforms of $f$ are given by
\begin{align*}
\Wc_{\text{Rot}} f &:= \{f \ast \psi_{\lambda, R} (x): x \in \mathbb{R}^n, \lambda \in (0, \infty), R \in \text{SO}(n)\},\\
W_\text{Rot}f &:= \{f \ast \psi_{j, R} (x): x \in \mathbb{R}^n, j \in \mathbb{Z}, R \in \text{SO}(n)\}.
\end{align*}
We will first consider a translation invariant and rotation equivariant formulation of continuous and dyadic one-layer scattering using
\begin{align*}
\mathfrak{S}_{\text{cont},q}f(\lambda, R) &:= \| f \ast \psi_{\lambda, R}\|_{q}, \\  
\mathfrak{S}_{\text{dyad},q}f(j, R) &:= \| f \ast \psi_{j, R}\|_{q}.
\end{align*}

The translation invariance of our representation follows from translation invariance of the norm. For rotation equivariance, notice that if $f_{\tilde R}(x) := f(\tilde R^{-1}x)$, then we have
\begin{align*}
\mathfrak{S}_{\text{cont},q} f_{\tilde R}(\lambda, R) &= \mathfrak{S}_{\text{cont},q} f(\lambda, \tilde{R}^{-1}R), \\
\mathfrak{S}_{\text{dyad},q}f_{\tilde R}(j, R) &= \mathfrak{S}_{\text{dyad},q} f(j, \tilde{R}^{-1}R). 
\end{align*}
Now suppose we have $m$ layers again. Then we define our $m$ layer transforms by
\begin{align*}
\mathfrak{S}_{\text{cont},q}^mf(\lambda_1, \ldots, \lambda_m, R_1, \ldots, R_m)&:= \| |f \ast \psi_{\lambda_1, R_1}| \ast \dots |\ast \psi_{\lambda_m, R_m}\|_{q},\\
\mathfrak{S}_{\text{dyad},q}^mf(j_1, \ldots, j_m, R_1, \ldots, R_m)&:= \| |f \ast \psi_{j_1, R_1}| \ast \dots |\ast \psi_{j_m, R_m}\|_{q}.
\end{align*}
and rotation equivariance implies 
\begin{align*}
\mathfrak{S}_{\text{cont},q}^mf_{\tilde{R}}(\lambda_1, \ldots, \lambda_m, R_1, \ldots, R_m) &= \mathfrak{S}_{\text{cont},q}^mf(\lambda_1, \ldots, \lambda_m, \tilde{R}^{-1}R_1, \ldots, \tilde{R}^{-1}R_m), \\
\mathfrak{S}_{\text{dyad},q}^mf_{\tilde{R}}(j_1, \ldots, j_m, R_1, \ldots, R_m) &= \mathfrak{S}_{\text{dyad},q}^mf(j_1, \ldots, j_m, \tilde{R}^{-1}R_1, \ldots, \tilde{R}^{-1}R_m).
\end{align*}

The norm we will use is similar to our previous formulations. Denote the scattering norm for the continuous transform as
\begin{equation*}
\resizebox{\textwidth}{!}{
$\|\mathfrak{S}_{\text{cont},q}^mf\|_{\Lb^2(\mathbb{R}_+^m) \times \text{SO}(n)^m}^q := \left(\int_0^{\infty}\int_{\text{SO}(n)}\dots\int_0^{\infty}\int_{\text{SO}(n)}  \| |f \ast \psi_{j_1, R_1}| \ast \dots |\ast \psi_{j_m, R_m}\|^{2}_{q} d\mu_1(R_1)\, \frac{d\lambda_1}{\lambda_1^{n+1}} \dots d\mu_m(R_n)\,\frac{d\lambda_m}{\lambda_m^{n+1}}.\right)^{q/2}$}
\end{equation*}
For the dyadic transform, we denote the norm using 
$$\|\mathfrak{S}_{\text{dyad},q}^mf\|_{\ellb^2(\Z^m) \times \text{SO}(n)^m}^q := \left(\sum_{j_m \in \mathbb{Z}}\int_{\text{SO}(n)}\dots\sum_{j_1 \in \mathbb{Z}}\int_{\text{SO}(n)} \| |f \ast \psi_{j_1, R_1}| \ast \dots |\ast \psi_{j_m, R_m}\|^{2}_{q} d\mu_1(R_1)\, \dots d\mu_m(R_n)\right)^{q/2}.$$
We will start by proving that these formulations of the scattering transform are well defined, and prove properties about stability to diffeomorphisms like in previous sections.
\begin{lemma}
Let $\psi$ be a wavelet that satisfies properties \eqref{eqn: decay condition} and \eqref{eqn: holder condition}.
\begin{itemize}
    \item If $1 < q \leq 2$, we have $\mathfrak{S}_{\text{cont},q}^m:\Lb^q(\mathbb{R}^n) \rightarrow \Lb^2(\mathbb{R}^m_+) \times{\text{SO}(n)^m}$ and 
$\mathfrak{S}_{\text{dyad},q}^m:\Lb^q(\mathbb{R}^n) \rightarrow \ell^2(\mathbb{Z}^m) \times{\text{SO}(n)^m}$. 
     \item If $q  = 1$ and one of the following holds:
    \begin{itemize}
        \item $n = 1$ and $\psi$ is complex analytic,
        \item $n \geq 2$ and $\psi$ satisfies the conditions of Lemma \ref{lem: riesz transform decay},
    \end{itemize} then $\mathfrak{S}_{\text{cont},1}^m:\Lb^1(\mathbb{R}^n) \rightarrow \Lb^2(\mathbb{R}^m_+) \times{\text{SO}(n)^m}$ and 
$\mathfrak{S}_{\text{dyad},1}^m:\Lb^1(\mathbb{R}^n) \rightarrow \ell^2(\mathbb{Z}^m) \times{\text{SO}(n)^m}$.
    \item If $\psi$ is also a Littlewood-Paley wavelet, we have
    \begin{align*}
    \|\mathfrak{S}_{\text{cont},2}^mf\|_{\Lb^2(\mathbb{R}_+^m) \times \text{SO}(n)^m}^2 &= \mu(\text{SO}(n))^m C_\psi^m \|f\|_2^2,\\
    \|\mathfrak{S}_{\text{dyad},q}^mf\|_{\ellb^2(\mathbb{Z}^m) \times \text{SO}(n)^m}^2 &= \mu(\text{SO}(n))^m \hat{C}_\psi^m \|f\|_2^2.
    \end{align*}
\end{itemize}
\end{lemma}
\begin{proof}
We prove the first and third claim. The second claim is almost identical to the first claim, so the proof will be omitted for brevity. Note that we will only provide arguments for the continuous scattering transform since the proofs for the dyadic transform are very similar. By Fubini Theorem and boundedness of the $m$-layer scattering transform, there exists a constant $C_q >0$, which is dependent on $q$, such that 
\begin{align*}
&\|\mathfrak{S}_{\text{cont},q}^mf\|_{\Lb^2(\mathbb{R}_+^m) \times \text{SO}(n)^m}^q \\
&= \left[\int_0^{\infty}\int_{\text{SO}(n)}\dots\int_0^{\infty}\int_{\text{SO}(n)}  \| |f \ast \psi_{\lambda_1, R_1}| \ast \dots |\ast \psi_{\lambda_m, R_m}\|^{2}_{q} d\mu(R_m)\, \frac{d\lambda_1}{\lambda_1^{n+1}} \cdots d\mu(R_1)\,\frac{d\lambda_m}{\lambda_m^{n+1}}\right]^{q/2} \\
&= \left[\int_{\text{SO}(n)}\dots \int_{\text{SO}(n)} \left(\int_0^{\infty} \cdots \int_0^{\infty}  \| |f \ast \psi_{\lambda_1, R_1}| \ast \dots |\ast \psi_{\lambda_m, R_m}\|^{2}_{q} \,\frac{d\lambda_1}{\lambda_1^{n+1}} \cdots \frac{d\lambda_m}{\lambda_m^{n+1}}\right)^{\frac{q}{2}\cdot \frac{2}{q}} \, d\mu(R_1)\,  \cdots d\mu(R_m)\right]^{q/2} \\
&\leq  \left[\int_{\text{SO}(n)} \cdots \int_{\text{SO}(n)} (C_q^{mq} \|f\|_q^q)^{2/q} \, d\mu(R_1) \cdots d\mu(R_m)\right]^{q/2}\\
&= C_q^{mq} \mu(\text{SO}(n))^{mq/2} \|f\|_q^q
\end{align*} because each $\psi_{\lambda_i, R_i}$ is still a wavelet with sufficient decay even if the rotation is applied. For the third claim, we see that
\begin{align*}
&\|\mathfrak{S}_{\text{cont},2}^mf\|_{\Lb^2(\mathbb{R}_+^m) \times \text{SO}(n)^m}^2 \\
&=\int_{\text{SO}(n)}\cdots\int_{\text{SO}(n)}\left(\int_0^{\infty}\cdots\int_0^{\infty}\| |f \ast \psi_{\lambda_1, R_1}| \ast \dots |\ast \psi_{\lambda_m, R_m}\|^2_{\Lb^2(\mathbb{R}^n)}  \frac{d\lambda_1}{\lambda_1^{n+1}}\cdots \frac{d\lambda_m}{\lambda_m^{n+1}}\right) d\mu(R_1) \cdots d\mu(R_n) \\
&=  \int_{\text{SO}(n)} \cdots \int_{\text{SO}(n)} C_\psi ^m \|f\|_2^2 \, d\mu(R_1) \cdots d\mu(R_m)\\
&=  \mu(\text{SO}(n))^m C_\psi^m\|f\|_2^2.
\end{align*}
\end{proof}

\begin{theorem}
Assume $|c|< \tfrac{1}{2n}$. Let $\tau(x) = cx$ and $L_\tau f(x) = f((1-c)x)$. Suppose that $\psi$ is a wavelet that satisfies the conditions of Lemma \ref{prop: newwavelet}. Then there exist constants $\tilde{K}_{n,m,q}$ and $\tilde{K}_{n,m,q}'$ dependent only on $n$, $m$, and $q$ such that
\begin{equation*}
    \|\mathfrak{S}_{\text{cont},q}^mf - \mathfrak{S}_{\text{cont},q}^mL_\tau f \|_{\Lb^2(\R_+^m) \times \text{SO}(n)^m}^q \leq |c|^q \cdot \tilde{K}_{n,m,q} \|f\|_q^q
\end{equation*}
and 
\begin{equation*}
    \|\mathfrak{S}_{\text{dyad},q}^mf - \mathfrak{S}_{\text{dyad},q}^mL_\tau f \|_{\ellb^2(\Z^m) \times \text{SO}(n)^m}^q \leq |c|^q \cdot \tilde{K}_{n,m,q}' \|f\|_q^q.
\end{equation*}
Alternatively, if one of the following holds:
\begin{itemize}
    \item $n = 1$, $\psi$ is complex analytic and satisfies the conditions of Lemma \ref{prop: newwavelet},
    \item $n \geq 2$ and $\psi$ satisfies the conditions of Lemma \ref{lem: riesz transform decay},
\end{itemize} there exist $\tilde{H}_{m,n}$ and $\tilde{H}_{m,n}'$ such that
\begin{equation*}
    \|\mathfrak{S}_{\text{cont},1}^mf - \mathfrak{S}_{\text{cont},1}^mL_\tau f \|_{\Lb^2(\R_+^m) \times \text{SO}(n)^m} \leq |c| \cdot \tilde{H}_{m,n} \|f\|_{\mathbb{H}^1(\mathbb{R}^n)}.
\end{equation*}
and 
\begin{equation*}
    \|\mathfrak{S}_{\text{dyad},1}^mf - \mathfrak{S}_{\text{dyad},1}^mL_\tau f \|_{\ellb^2(\Z^m) \times \text{SO}(n)^m} \leq |c| \cdot \tilde{H}_{m,n}' \|f\|_{\mathbb{H}^1(\mathbb{R}^n)}
\end{equation*}
\end{theorem}

\begin{theorem}
Let $\tau \in C^2(\mathbb{R}^n)$, and let $L_\tau f(x) = f(x - \tau(x))$. Suppose that $\psi$ is a wavelet such that the wavelet and all its first and second partial derivatives have $O((1+|x|)^{-n-3})$ decay. When $q \in (1,2)$, there exists a constant $C_{n,m,q}$ dependent on $\mu(\text{SO}(n))$, $n$, $m$, and $q$ such that
\begin{equation*}
\resizebox{\textwidth}{!}{
    $\|\mathfrak{S}_{\text{cont},q}^mf - \mathfrak{S}_{\text{cont},q}^mL_\tau f \|_{\Lb^2(\R_+^m) \times \text{SO}(n)^m}^q \leq C_{n,m,q} \left[ \|D\tau\|_\infty^q + \eta^q M_r^{q\delta}  \left(\|D\tau\|_\infty \left(\log\frac{\|\Delta \tau\|_\infty}{\|D\tau\|_\infty} \vee 1\right) + \|D^2\tau\|_\infty \right)^{q(1-\delta)}\right] \|f\|_q^q,$}
\end{equation*}
\begin{equation*}
    \resizebox{\textwidth}{!}{
$\|\mathfrak{S}_{\text{dyad},q}^mf - \mathfrak{S}_{\text{dyad},q}^mL_\tau f \|_{\ellb^2(\Z^m) \times \text{SO}(n)^m}^q \leq \tilde{C}_{n,m,q}\left[ \|D\tau\|_\infty^q + \eta^q M_r^{q\delta}  \left(\|D\tau\|_\infty \left(\log\frac{\|\Delta \tau\|_\infty}{\|D\tau\|_\infty} \vee 1\right) + \|D^2\tau\|_\infty \right)^{q(1-\delta)}\right] \|f\|_q^q,$}
\end{equation*}
\begin{equation*}
    \|\mathfrak{S}_{\text{cont},2}^mf - \mathfrak{S}_{\text{cont},2}^mL_\tau f \|_{\Lb^2(\R_+^m) \times \text{SO}(n)^m}^2 \leq C_{n,m} \left[ \|D\tau\|_\infty^2  +   \left(\|D\tau\|_\infty \left(\log\frac{\|\Delta \tau\|_\infty}{\|D\tau\|_\infty} \vee 1\right) + \|D^2\tau\|_\infty \right)^{2}\right] \|f\|_2^2,
\end{equation*}
\begin{equation*}
    \|\mathfrak{S}_{\text{dyad},2}^mf -\mathfrak{S}_{\text{dyad},2}^mL_\tau f \|_{\ell^2(\Z^m) \times \text{SO}(n)^m}^2 \leq \tilde{C}_{n,m} \left[ \|D\tau\|_\infty^2 +   \left(\|D\tau\|_\infty \left(\log\frac{\|\Delta \tau\|_\infty}{\|D\tau\|_\infty} \vee 1\right) + \|D^2\tau\|_\infty \right)^{2}\right] \|f\|_2^2.
\end{equation*}
\end{theorem}

\subsection{Rotation Invariant Representations}
The representation before was rotation equivariant, but in some tasks, we would rather have rotation invariance. In \cite{mallat2012group}, the authors choose to integrate over each group action in a group of transformations. However, this will remove the information the relative angles between each action if we have multiple layers in our transform. 

In the case of one layer, since there is only one angle, we use a similar formulation to \cite{mallat2012group} and define continuous and dyadic scattering transforms for rotation invariance as 
\begin{align*}
\mathscr{S}_{\text{cont},q} f(\lambda) &= \int_{\text{SO}(n)} \|f \ast \psi_{\lambda, R}\|^q_{\Lb^q(\mathbb{R}^n)} d\mu(R),\\
\mathscr{S}_{\text{dyad},q} f(j) &= \int_{\text{SO}(n)} \|f \ast \psi_{j, R}\|^q_{\Lb^q(\mathbb{R}^n)} d\mu(R).
\end{align*}
The corresponding norms are given by
\begin{align*}
\|\mathscr{S}_{\text{cont},q}f\|_{\Lb^2(\mathbb{R}_+)}^q &:= \left[\int_0^\infty \left[\int_{\text{SO}(n)} \|f \ast \psi_{\lambda, R}\|_{q} \mu(R)\right]^{2/q} \, \frac{d \lambda}{\lambda^{n+1}}\right]^{q/2}, \\
\|\mathscr{S}_{\text{dyad},q}f\|_{\ell^2(\mathbb{Z})}^q &:= \left[\sum_{j \in \mathbb{Z}}\left[\int_{\text{SO}(n)} \|f \ast \psi_{j, R}\|_{q} \mu(R)\right]^{2/q}\right]^{q/2}.
\end{align*}

Now we generalize to the case where $m \geq 2.$ Let $R_1, \ldots, R_m \in \text{SO}(n)$. Define
%fix here. Should have R_1 by itself and then R_2R_1 etc...
%all should be a function of R_2 to R_m
\begin{align*}
\mathscr{S}_{\text{cont},q}^mf(\lambda_1, \ldots, \lambda_m, R_2, \ldots, R_m)&:= \int_{\text{SO}(n)}\| |f \ast \psi_{\lambda_1, R_2R_1}| \ast \dots \ast |\psi_{\lambda_m, R_mR_1}\|^2_{q} \,d\mu(R_1),\\
\mathscr{S}_{\text{dyad},q}^mf(j_1, \ldots, j_m, R_2, \ldots, R_m)&:= \int_{\text{SO}(n)} \| |f \ast \psi_{j_1, R_2R_1}| \ast \dots |\ast \psi_{j_m, R_mR_1}\|^2_{q}\,  d\mu(R_1).
\end{align*}
The norm for the continuous transform, the norm $\|\mathscr{S}_{\text{cont},q}^mf\|_{\Lb^2(\mathbb{R}_+^m) \times \text{SO}(n)^{m-1}}^q$, is given by
$$\left(\int_0^{\infty}\int_{\text{SO}(n)} \cdots \int_0^{\infty}\int_{\text{SO}(n)} \int_0^{\infty}\mathscr{S}_{\text{cont},q}^mf(\lambda_1, \ldots, \lambda_m, R_2, \ldots, R_m) \frac{d\lambda_1}{\lambda_1^{n+1}} \, d\mu_2(R_2)\,\frac{d\lambda_2}{\lambda_2^{n+1}} \dots d\mu_m(R_m)\,\frac{d\lambda_m}{\lambda_m^{n+1}}\right)^{q/2}$$
For the dyadic transform, the norm $\|\mathscr{S}_{\text{dyad},q}^mf\|_{\ell^2(\mathbb{Z}) \times \text{SO}(n)^{m-1}}^q$ is given by
$$\left(\sum_{j_m \in \mathbb{Z}}\int_{\text{SO}(n)}\dots\sum_{j_2 \in \mathbb{Z}}\int_{\text{SO}(n)}\sum_{j_1 \in \mathbb{Z}}\mathscr{S}_{\text{dyad},q}^mf(\lambda_1, \ldots, \lambda_m, R_2, \ldots, R_m)  d\mu_1(R_1) \, d\mu_2(R_2)\, \dots d\mu_m(R_m)\right)^{q/2}.$$
Like before, we will discuss the well-definedness and stability of these operators to diffeomorphisms. The proofs will be omitted since they follow directly from the previous sections with minor modifications.

\begin{lemma}
Let $\psi$ be a wavelet that satisfies properties \eqref{eqn: decay condition} and \eqref{eqn: holder condition}.
\begin{itemize}
    \item If $1 < q \leq 2$, we have $\mathscr{S}_{\text{cont},q}^m:\Lb^q(\mathbb{R}^n) \rightarrow \Lb^2(\mathbb{R}^m_+) \times{\text{SO}(n)^{m-1}}$ and 
$\mathscr{S}_{\text{dyad},q}^m:\Lb^q(\mathbb{R}^n) \rightarrow \ell^2(\mathbb{Z}^m) \times{\text{SO}(n)^{m-1}}$. 
    \item If $q  = 1$ and one of the following holds:
    \begin{itemize}
        \item $n = 1$ and $\psi$ is complex analytic,
        \item $n \geq 2$ and $\psi$ satisfies the conditions of Lemma \ref{lem: riesz transform decay},
    \end{itemize} then $\mathscr{S}_{\text{cont},1}^m:\Lb^1(\mathbb{R}^n) \rightarrow \Lb^2(\mathbb{R}^m_+) \times{\text{SO}(n)^{m-1}}$ and 
$\mathscr{S}_{\text{dyad},1}^m:\Lb^1(\mathbb{R}^n) \rightarrow \ell^2(\mathbb{Z}^m) \times{\text{SO}(n)^{m-1}}$. 
    \item If $q = 2$ and $\psi$ is also a littlewood paley wavelet, we have $\|\mathscr{S}_{\text{dyad},2}^mf\|_{\ellb^1(\Z^m) \times \text{SO}(n)^{m-1}} = \mu(\text{SO}(n))^{m-1} C_\psi^m \|f\|_2^2$ and $\|\mathscr{S}_{\text{cont},2}^mf\|_{\Lb^1(\mathbb{R}_+^{m}) \times \text{SO}(n)^{m-1}} = \mu(\text{SO}(n))^{m-1} \hat{C}_\psi^m \|f\|_2^2$.
\end{itemize}
\end{lemma}

\begin{theorem}
Assume $|c| < \tfrac{1}{2n}$and $1 < q < 2$. Let $\tau(x) = cx$ and let $L_\tau f(x) = f((1-c)x)$. Suppose that $\psi$ is a wavelet that satisfies the conditions of Lemma \ref{prop: newwavelet}. Then there exist constants $\hat{K}_{n,m,q}$ and $\hat{K}_{n,m,q}'$ dependent only on $n$, $m$, and $q$ such that
\begin{equation*}
    \|\mathscr{S}_{\text{cont},q}^mf - \mathscr{S}_{\text{cont},q}^mL_\tau f \|_{\Lb^2(\R_+^m) \times \text{SO}(n)^{m-1}}^q \leq |c|^q \cdot \hat{K}_{n,m,q} \|f\|_q^q
\end{equation*}
and 
\begin{equation*}
    \|\mathscr{S}_{\text{dyad},q}^mf - \mathscr{S}_{\text{dyad},q}^mL_\tau f \|_{\ellb^2(\Z^m) \times \text{SO}(n)^{m-1}}^q \leq |c|^q \cdot \hat{K}_{n,m,q}' \|f\|_q^q.
\end{equation*}
Additionally, if $q = 1$ and one of the following holds:
\begin{itemize}
    \item $n = 1$, $\psi$ is complex analytic and satisfies the conditions of Lemma \ref{prop: newwavelet},
    \item $n \geq 2$ and $\psi$ satisfies the conditions of Lemma \ref{lem: riesz transform decay},
\end{itemize} there exist $\hat{H}_{m,n}$ and $\hat{H}_{m,n}'$ such that
\begin{equation*}
    \|\mathscr{S}_{\text{cont},1}^mf - \mathscr{S}_{\text{cont},1}^mL_\tau f \|_{\Lb^2(\R_+^m) \times \text{SO}(n)^{m-1}} \leq |c| \cdot \hat{H}_{m,n} \|f\|_{\mathbb{H}^1(\mathbb{R}^n)}
\end{equation*}
and 
\begin{equation*}
    \|\mathscr{S}_{\text{dyad},1}^mf - \mathscr{S}_{\text{dyad},1}^mL_\tau f \|_{\ellb^2(\Z^m) \times \text{SO}(n)^{m-1}} \leq |c| \cdot \hat{H}_{m,n}' \|f\|_{\mathbb{H}^1(\mathbb{R}^n)}.
\end{equation*}
\end{theorem}

\begin{theorem}
Let $\tau \in C^2(\mathbb{R}^n)$ and define $L_\tau f(x) = f(x - \tau(x))$ with $\|D\tau\|_\infty < \tfrac{1}{2n}$. Suppose that $\psi$ is a wavelet such that the wavelet and all its first and second partial derivatives have $O((1+|x|)^{-n-3})$ decay. For $q \in(1,2]$, there exist constants $C_{m,n}$, $\hat{C}_{m,n}$, $C_{m,n,q}$, and $\hat{C}_{m,n,q}$ such that
\begin{equation*}
    \|\mathscr{S}_{\text{cont},2}^mf - \mathscr{S}_{\text{cont},2}^mL_\tau f \|_{\Lb^2(\R_+^m)\times \text{SO}(n)^{m-1}}^2 \leq C_{m,n} \, \left(\norm{D\tau}_{\infty}^2+\left(\norm{D\tau}_{\infty}\left(\log\frac{\norm{\Delta\tau}_\infty}{\norm{D\tau}_\infty} \vee 1\right)+\norm{D^2\tau}_\infty\right)^2\right) \|f\|_2^2,
\end{equation*}
\begin{equation*}
    \|\mathscr{S}_{\text{dyad},2}^mf - \mathscr{S}_{\text{dyad},2}^mL_\tau f \|_{\ellb^2(\Z^m)\times \text{SO}(n)^{m-1}}^2 \leq \hat{C}_{m,n} \, \left(\norm{D\tau}_{\infty}^2+\left(\norm{D\tau}_{\infty}\left(\log\frac{\norm{\Delta\tau}_\infty}{\norm{D\tau}_\infty} \vee 1\right)+\norm{D^2\tau}_\infty\right)^2\right) \|f\|_2^2,
\end{equation*}
\begin{equation*}
    \resizebox{\textwidth}{!}{
    $\|\mathscr{S}_{\text{cont},q}f - \mathscr{S}_{\text{cont},q}^mL_\tau f \|_{\Lb^2(\R_+^m)\times \text{SO}(n)^{m-1}}^q \leq C_{m,n,q} \left[ \|D\tau\|_\infty^q + \eta^q M_r^{q\delta}  \left(\|D\tau\|_\infty \left(\log\frac{\|\Delta \tau\|_\infty}{\|D\tau\|_\infty} \vee 1\right) + \|D^2\tau\|_\infty \right)^{q(1-\delta)}\right] \|f\|_q^q,$}
\end{equation*}
\begin{equation*}
    \resizebox{\textwidth}{!}{
    $\|\mathscr{S}_{\text{dyad},q}^mf - \mathscr{S}_{\text{dyad},q}^mL_\tau f \|_{\ellb^2(\Z^m)\times \text{SO}(n)^{m-1}}^q \leq \hat{C}_{m,n,q} \, \left[ \|D\tau\|_\infty^q + \eta^q M_r^{q\delta}  \left(\|D\tau\|_\infty\left(\log\frac{\|\Delta \tau\|_\infty}{\|D\tau\|_\infty} \vee 1\right) + \|D^2\tau\|_\infty \right)^{q(1-\delta)}\right] \|f\|_q^q.$
    }
\end{equation*}

\end{theorem}

\section{Conclusion}
We have formulated operators that are translation invariant in $\Lb^q(\mathbb{R}^n)$, proven these operators are Lipschitz continuous to the action of $C^2$ diffeomorphisms when $1 < q \leq 2$ with respect to certain norms, and used these results to formulate rotation invariant/equivariant operators on $\Lb^q(\mathbb{R}^n)$ that are Lipschitz continuous to the action of $C^2$ diffeomorphisms. One question that was left unanswered was if Lipschitz continuity holds for general diffeomorphisms when $q = 1$. This question is harder to answer because $f \in \mathbf{H}^1(\mathbb{R}^n)$ does not necessarily imply that $L_\tau f\in \mathbf{H}^1(\mathbb{R}^n)$. The kernel for the commutator is also singular, which would mean one cannot use extension theorems for Hardy spaces. The answer is most likely no, but we did not construct a counterexample. 

\section{Acknowledgments}
We would like to thank Michael Perlmutter for providing suggestions that improved the clarity of our initial draft. We would also like to thank the anonymous reviewers for their careful feedback over multiple iterations. Lastly, we would like to thank Yang Yang for checking over a revision to the proof of Proposition 15.

\bibliographystyle{elsarticle-num}
\bibliography{paper.bib}

\end{document}